\documentclass{amsart}

\usepackage{amsmath,amssymb}
\usepackage{amsthm}
\usepackage{amsrefs}
\usepackage{indentfirst}
\usepackage{mathrsfs}
\usepackage{hyperref}
\usepackage{xcolor}
\usepackage[margin=1.4in]{geometry}
\usepackage{nicefrac}

\newcommand{\bR}{\mathbb{R}}

\newcommand{\norm}[1]{\left\Vert #1\right\Vert}

\newcommand{\calM}{\mathcal{M}}
\newcommand{\calT}{\mathcal{T}}
\newcommand{\calP}{\mathcal{P}}
\newcommand{\calG}{\mathcal{G}}

\newlength{\leftstackrelawd}
\newlength{\leftstackrelbwd}
\def\leftstackrel#1#2{\settowidth{\leftstackrelawd}%
	{${{}^{#1}}$}\settowidth{\leftstackrelbwd}{$#2$}%
	\addtolength{\leftstackrelawd}{-\leftstackrelbwd}%
	\leavevmode\ifthenelse{\lengthtest{\leftstackrelawd>0pt}}%
	{\kern-.5\leftstackrelawd}{}\mathrel{\mathop{#2}\limits^{#1}}}

\numberwithin{equation}{section}
\newtheorem{theorem}{Theorem}[section]
\newtheorem{lemma}[theorem]{Lemma}

\newtheorem{assumption}[theorem]{Assumption}

\allowdisplaybreaks[4]

\title[the Gross-Pitaevskii eigenvalue problem]{On the convergence of Sobolev gradient flow for the Gross-Pitaevskii eigenvalue problem}
\author{Ziang Chen}
\address{(ZC) Department of Mathematics, Massachusetts Institute of Technology, 77 Massachusetts Avenue, Cambridge, MA 02139.}
\email{ziang@mit.edu}
\author{Jianfeng Lu}
\address{(JL) Departments of Mathematics, Physics, and Chemistry, Duke University, Box 90320, Durham, NC 27708.}
\email{jianfeng@math.duke.edu}
\author{Yulong Lu}
\address{(YL) School of Mathematics, University of Minnesota, 206 Church Street SE, Minneapolis, MN 55455.}
\email{yulonglu@umn.edu}
\author{Xiangxiong Zhang}
\address{(XZ) Department of Mathematics, Purdue University, 150 N. University Street, West Lafayette, IN 47907.}
\email{zhan1966@purdue.edu}
\date{\today}

\begin{document}
	\begin{abstract}
		We study the convergences of three projected Sobolev gradient flows to the ground state of the Gross-Pitaevskii eigenvalue problem. They are constructed as the gradient flows of the Gross-Pitaevskii energy functional with respect to the $H^1_0$-metric and two other equivalent metrics on $H_0^1$, including the iterate-independent $a_0$-metric and the iterate-dependent $a_u$-metric. We first prove the energy dissipation property and the global convergence to a critical point of the Gross-Pitaevskii energy for the discrete-time $H^1$ and $a_0$-gradient flow. We also prove local exponential convergence of all three schemes to the ground state. 
	\end{abstract}
	
	\maketitle
	
	\section{Introduction}
	
	This paper concerns the following nonlinear Schr\"odinger eigenvalue problem, also known as the Gross-Pitaevskii \cites{gardiner1997particle, atre2006class, dast2013eigenvalue, rogel2013gross} eigenvalue problem:
	\begin{equation}\label{eigen_problem}
		\begin{cases}
			-\Delta u + V u + \beta |u|^2 u = \lambda u, & \text{in }\Omega,\\
			u = 0, &\text{on }\partial\Omega,
		\end{cases}
	\end{equation}
	where $\Omega$ is a bounded Lipschitz domain on $\bR^d$ ($d = 1,2,3$), $V:\Omega\rightarrow\bR$ is a non-negative potential energy, and $\beta\geq 0$.  The Gross-Pitaevskii eigenvalue problem has been widely used in the quantum physics community to represent the Bose-Einstein condensation \cites{bose1924plancks,einstein1925quantentheorie,dalfovo1999theory,pitaevskii2003bose}. The wavefunction associated with a stationary state of the system can be described by an eigenfunction $u$ to \eqref{eigen_problem}, with the eigenvalue $\lambda$ being the chemical potential. 
	
	Many numerical methods have been proposed to compute the ground state of the problem \eqref{eigen_problem}, i.e., the $L^2$-normalized eigenfunction corresponding to the smallest eigenvalue. Among them, one of the most popular classes of methods is the discrete normalized gradient flow \cite{bao2004computing} which applies the backward Euler time-discretization for the continuous $L^2$-gradient flow of the constrained energy. Several alternative gradient flows have been designed based on the idea of tuning the geometry of the gradient flows, including the projected Sobolev gradient flow \cites{danaila2010new, kazemi2010minimizing, danaila2017computation, zhang2019exponential, henning2020sobolev, heid2021gradient} and the J-method \cites{jarlebring2014inverse, altmann2021j}. Apart from the gradient-flow-based methods, self-consistent field iteration \cites{cances2000scf,cances2000convergence,upadhyaya2018density} is another class of methods that solves the nonlinear eigenvalue problem by iteratively solving a series of linearized eigenvalue problems. There have been other works investigating numerical methods for the Gross-Pitaevskii equation, such as the analysis of finite-dimensional approximation \cite{zhou2003analysis}, error estimate \cite{cances2010numerical}, two-grid method \cite{chien2006two}, and multigrid method \cites{xie2016multigrid, zhang2019efficient}, just to name a few. Finally, let us also mention some analytical studies of the Gross-Pitaevskii equation/eigenvalue problem, e.g., the formal analytical solutions \cite{trallero2008formal}, the stability \cite{jackson2005stability}, and the posteriori analysis \cite{dusson2017posteriori}.
	
	In this work, we focus on the projected Sobolev gradient flow approach to computing the ground state of \eqref{eigen_problem}. More specifically, the ground state of \eqref{eigen_problem} is characterized by the minimizer of the following variational problem:
	\begin{equation}\label{min_energy_problem}
		\min_{u\in H_0^1(\Omega)} E(u) := \int_\Omega \frac{1}{2}|\nabla u|^2 +\frac{1}{2} V|u|^2 +\frac{\beta}{4} |u|^4, \quad\text{s.t. } \norm{u}_{L^2(\Omega)} = 1.
	\end{equation}
	Motivated by Riemannian optimization, i.e., optimization subject to a Riemannian manifold constraint, finding the ground state of \eqref{eigen_problem} is equivalent to minimizing the energy functional $E(u)$ over an infinite dimensional Hilbert manifold $\calM$  in $H_0^1(\Omega)$ defined by 
	\begin{equation*}
		\calM = \left\{u\in H_0^1(\Omega):\norm{u}_{L^2(\Omega)} = 1\right\}.
	\end{equation*}
	The projected Sobolev gradient flow for solving \eqref{eigen_problem}  is defined as
	\begin{equation}\label{eq:psgf}
		u^\prime(t) = - \nabla_X^{\mathcal{R}} E(u(t)) = -\mathcal{P}_{\calT_u \calM ,X} (\nabla_{X} E(u(t))),
	\end{equation}
	where $\nabla_X^{\mathcal{R}} E(u(t))$ is the Riemannian gradient of $E$ associated with the inner product $(\cdot,\cdot)_X$ and the manifold $\calM$, i.e., the projection of the gradient $\nabla_{X} E(u(t))$ onto the tangent space $\calT_u \calM$. We are mainly interested in three projected Sobolev gradient flows that correspond to \eqref{eq:psgf} with different choices of inner products $X$: (i) the projected $H^1$-gradient flow \cite{kazemi2010minimizing} where  $(z,w)_{X} = (z,w)_{H_0^1(\Omega)} = (\nabla z, \nabla w)_{L^2(\Omega)}$; (ii) the projected $a_0$-gradient flow \cite{danaila2010new} where $(z,w)_{X} = (z,w)_{a_0(\Omega)} = \int_{\Omega} \nabla z\cdot \nabla w + V z w$; and (iii) the projected $a_u$-gradient flow \cite{henning2020sobolev} where
	$$(z,w)_{X} = (z,w)_{a_u(\Omega)} = \int_{\Omega} \nabla z\cdot \nabla w + V z w + \beta |u|^2 z w.$$
	The primary goal of the paper is to prove the convergence property of the time-discretization of the three projected Sobolev gradient flows above.

	\subsection*{Prior work and our contribution}
	The work \cite{kazemi2010minimizing} established the global exponential convergence of the continuous projected $H^1$-gradient flow to a critical point of $E$. However, the convergence analysis of its discrete version remained largely open. In fact, as stated in \cite{henning2020sobolev}, even the energy decay and the convergence to critical points was unclear for the discrete projected $H^1$-gradient flow. In addition, the convergence of the discrete projected $a_0$-gradient flow has not been proved in previous literature, though conjectured to be true from numerical studies. One major contribution of this work is to prove the energy decay as well as the global convergence to a critical point for the discrete projected $H^1$ and $a_0$-gradient flow. Moreover, for the same schemes, we also obtained the local exponential convergence rate. 
	
	As for the projected $a_u$-gradient flow, the work \cite{henning2020sobolev} obtained a global exponential convergence of the continuous flow to the ground state and also proved the global convergence (without a rate) of its forward Euler discretization. In the recent work \cite{zhang2019exponential}, the author obtained a local exponential convergence of the discrete projected $a_u$-gradient flow under the assumption that the discrete iterates are uniformly bounded in $L^\infty(\Omega)$. However, such an assumption is very difficult to verify in practice. The second contribution of the current paper is to prove the local exponential convergence of the discrete projected $a_u$-gradient flow without such an assumption.
	
	Let us also mention several previous works on minimizing related energy functionals with more general spherical constraints. The local convergence for Hartree-Fock and Kohn-Sham functionals over Grassmann manifold was established in \cite{schneider2009direct}. In \cite{altmann2022energy}, the authors generalized the approach in \cite{zhang2019exponential} to show a convergence result for Kohn-Sham functional over Stiefel manifold.

	\subsection*{Organization} The rest of this paper will be organized as follows. In Section \ref{sec:gf}, we review three discrete projected Sobolev gradient flows. In Section  \ref{sec:main} we state the main convergence results on the three gradient flows. Sections \ref{sec:proofglobal}-\ref{sec:prooflocal} are devoted to the proofs of the main convergence results. The appendix contains some useful lemmas and auxiliary proofs of the main results.
	
	\section{Projected Sobolev Gradient Flow}\label{sec:gf}
	This section reviews the projected Sobolev gradient descent for solving \eqref{eigen_problem} or equivalently \eqref{min_energy_problem}, following \cites{zhang2019exponential, henning2020sobolev}; we refer to \cites{zhang2019exponential, henning2020sobolev} for more details. We remark that discrete flows/schemes considered in this paper are discrete only in time and there is no spatial discretization in the scheme. The spatial discretization of the discrete projected Sobolev gradient flow was given in \cite{henning2020sobolev} and references therein.
	
	We would assume that $V\in L^\infty(\Omega)$ and without loss of generality, we further assume that $0\leq V\leq V_{\max}< \infty$. Notice that since we consider the dimension $d \in \{1,2,3\}$, the embedding $H_0^1(\Omega)\subset L^4(\Omega)$ holds and the energy $E(u)$ defined in \eqref{min_energy_problem} is finite for any $u\in H_0^1(\Omega)$. The tangent space at the base $u\in \calM$ is given by  
	\begin{equation*}
		\calT_u\calM = \left\{\xi\in H_0^1(\Omega):(u,\xi)_{L^2(\Omega)} = 0\right\}.
	\end{equation*}
	Recall that the problem \eqref{min_energy_problem} can be viewed as an optimization problem on the manifold $\calM$, for which it is natural to consider the projected gradient method, i.e., update against the Riemannian gradient direction in the tangent space and then retract the iterate back to the manifold. The retraction map is clear in this setting:
	\begin{equation*}
		R(u) = \frac{u}{\norm{u}_{L^2(\Omega)}}\in\calM,\quad \forall~u\in H_0^1(\Omega)\backslash\{0\}.
	\end{equation*}
	However, the Riemannian gradient would depend on the inner product we equip at $u\in H_0^1(\Omega)$, which leads to different schemes, including $H^1$-scheme \cite{kazemi2010minimizing}, $a_0$-scheme \cite{danaila2010new}, and $a_u$-scheme \cite{henning2020sobolev}, that are described in the following subsections.
	
	\subsection{$H^1$-scheme} Let $H_0^1(\Omega)$ be equipped with the inner product $(u,v)_{H_0^1(\Omega)}:= (\nabla u,\nabla v)_{L^2(\Omega)}$. For any $w\in H_0^1(\Omega)$, by Riesz representation theorem, there exists a unique $\calG_{H^1}w\in H_0^1(\Omega)$ such that
	\begin{equation*}
		\left(z, \calG_{H^1} w\right)_{H_0^1(\Omega)}  = (z,w)_{L^2(\Omega)},\quad\forall~z\in H_0^1(\Omega).
	\end{equation*}
	$\calG_{H^1}: H_0^1(\Omega)\rightarrow H_0^1(\Omega)$ is named as the Green's operator. In other words, $u = \calG_{H^1}w$ is the unique solution to $ - \Delta u = w$ in $\Omega$ with the boundary condition $u=0$ on $\partial\Omega$.
	
	For any $u \in \calM$ and $h \in \calT_u \calM$, it holds that
	\begin{align*}
		( \nabla_{H^1} E(u), h)_{H_0^1(\Omega)}  & = \lim_{t\rightarrow 0} \frac{ E(u + th) - E(u)}{t}\\
		& = \lim_{t\rightarrow 0} \frac{1}{t} \bigg(\int_\Omega \left(\frac{1}{2}|\nabla u + t\nabla h|^2 + \frac{1}{2} V | u+th |^2 +\frac{\beta}{4} | u+th |^4\right) \\
		&\qquad \qquad\qquad - \int_\Omega \left(\frac{1}{2} |\nabla u|^2 + \frac{1}{2} V |u|^2 + \frac{\beta}{4} |u|^4\right)\bigg)\\
		&= (\nabla u,\nabla h)_{L^2(\Omega)} + (Vu + \beta |u|^2 u, h)_{L^2(\Omega)}\\
		& = \left(u + \calG_{H^1}( Vu + \beta |u|^2 u ), h\right)_{H_0^1(\Omega)},
	\end{align*}
	which implies that the $H^1$-gradient of the energy $E(u)$ can be evaluated as
	\begin{equation}\label{nabla_E_H1}
		\nabla_{H^1} E(u) = u + \calG_{H^1}( Vu + \beta |u|^2 u ).
	\end{equation}
	The lemma below computes the projection of any $\xi\in \calM$ on the tangent space $\calT_u \calM$.

	\begin{lemma}\label{lem:proj_H1}
		Given $u\in \calM$ and $\xi\in H_0^1(\Omega)$, the projection of $\xi$ onto $\calT_u\calM$, with respect to the $H_0^1$-inner product, is
		\begin{equation}\label{proj_H1}
			\calP_{\calT_u\calM, H^1}(\xi) = \xi - \frac{(\xi, u)_{L^2(\Omega)}}{\norm{\calG_{H^1} u}_{H_0^1(\Omega)}^2} \calG_{H^1} u.
		\end{equation}
	\end{lemma}
	
	\begin{proof}
		Let $\calP_{\calT_u\calM}(\xi)$ be defined as in \eqref{proj_H1}. First, $\calP_{\calT_u\calM}(\xi)\in \calT_u\calM$ since
		\begin{align*}
			\left(\calP_{\calT_u\calM}(\xi) , u\right)_{L^2(\Omega)} & = (\xi, u)_{L^2(\Omega)} - \frac{(\xi, u)_{L^2(\Omega)}}{\norm{\calG_{H^1} u}_{H_0^1(\Omega)}^2} \left(\calG_{H^1} u, u\right)_{L^2(\Omega)} \\ & = (\xi, u)_{L^2(\Omega)} - (\xi, u)_{L^2(\Omega)} = 0.
		\end{align*}
		In addition, for any $\eta\in \calT_u\calM$, it follows from $\left( \calG_{H^1} u,\eta\right)_{H_0^1(\Omega)} = (u,\eta)_{L^2(\Omega)} = 0$ that
		\begin{equation*}
			\left( \calP_{\calT_u\calM}(\xi),\eta\right)_{H_0^1(\Omega)} = (\xi, \eta)_{H_0^1(\Omega)}.
		\end{equation*}
		Therefore, $\calP_{\calT_u\calM}(\xi)$ is the desired projection.
	\end{proof}
	
	Combining \eqref{nabla_E_H1} and \eqref{proj_H1}, we obtain the Riemannian gradient of $E$ with respect to $H_0^1$-inner product at $u\in\calM$:
	\begin{align*}
		\nabla^{\mathcal{R}}_{H^1} E(u) & = \calP_{\calT_u\calM, H^1}(\nabla_{H^1} E(u)) =  \calP_{\calT_u\calM, H^1}\left(u + \calG_{H^1} (V u + \beta |u|^2 u)\right)\\
		& = u + \calG_{H^1} (V u + \beta |u|^2 u) - \frac{1 + \left(\calG_{H^1} (V u + \beta |u|^2 u), u\right)_{L^2(\Omega)}}{\norm{\calG_{H^1} u}_{H_0^1(\Omega)}^2} \calG_{H^1} u.
	\end{align*}
	With the above, the projected $H^1$-gradient descent is given by 
	\begin{equation}\label{H1_scheme}
		u_{n+1} = R \left(u_n - \alpha\ \nabla^{\mathcal{R}}_{H^1} E(u_n) \right),\quad n= 1,2\cdots,
	\end{equation}
	where $\alpha$ is the stepsize. 
	
	\subsection{$a_0$-scheme} Another choice of the inner product on $H_0^1(\Omega)$ is 
	\begin{equation*}
		(z,w)_{a_0(\Omega)} = \int_\Omega \nabla z\cdot\nabla w + V z w,
	\end{equation*}
	whose induced norm $\norm{\cdot}_{a_0(\Omega)}$ is equivalent to $\norm{\cdot}_{H_0^1(\Omega)}$ since $V\in L^\infty(\Omega)$ and is non-negative. 
	Similarly, there exists a unique Green's operator $\calG_{a_0}: H_0^1(\Omega)\rightarrow H_0^1(\Omega)$ such that
	\begin{equation*}
		\left(z, \calG_{a_0} w\right)_{a_0(\Omega)}  = (z,w)_{L^2(\Omega)},\quad\forall z\in H_0^1(\Omega),
	\end{equation*}
	and the $a_0$-gradient of $E(u)$ is
	\begin{equation*}
		\nabla_{a_0} E(u) =  u + \beta \calG_{a_0}( |u|^2 u ).
	\end{equation*}
	The project of $\xi\in H_0^1(\Omega)$ onto $\calT_u\calM$, where $u\in H_0^1(\Omega)$, with respect to the $a_0$-inner product, also reads similarly:
	\begin{equation*}
		\calP_{\calT_u\calM, a_0}(\xi) = \xi - \frac{(\xi, u)_{L^2(\Omega)}}{\norm{\calG_{a_0} u}_{a_0(\Omega)}^2} \calG_{a_0} u,
	\end{equation*}
	and hence the $a_0$-scheme iterates as
	\begin{equation}\label{a0_scheme}
		u_{n+1} = R \left(u_n - \alpha\ \nabla^{\mathcal{R}}_{a_0} E(u_n) \right),\quad n= 1,2\cdots,
	\end{equation}
	where
	\begin{align*}
		\nabla^{\mathcal{R}}_{a_0} E(u) & = \calP_{\calT_u\calM, a_0}(\nabla_{a_0} E(u)) = \calP_{\calT_u\calM, a_0}\left(u + \beta \calG_{a_0}(|u|^2 u )\right) \\
		& = u + \beta \calG_{a_0}(|u|^2 u ) - \frac{1 + \beta \left( \calG_{a_0}(|u|^2 u ), u \right)_{L^2(\Omega)}}{\norm{\calG_{a_0} u}_{a_0(\Omega)}^2} \calG_{a_0} u.
	\end{align*}
	
	\subsection{$a_u$-scheme} 
	Now we turn to the derivation of another projected gradient flow with respect to an inner product that varies in the base $u\in H_0^1(\Omega)$. One intuition is that one can have a derivative neater than \eqref{nabla_E_H1} when choosing the inner product carefully. More precisely, for a given $u\in H_0^1(\Omega)$ we consider the inner product $(\cdot,\cdot)_{a_u}$ defined by 
	\begin{equation*}
		(z,w)_{a_u(\Omega)} = \int_\Omega \nabla z\cdot\nabla w +  V zw + \beta |u|^2 zw,
	\end{equation*}
	with an associated Green's operator $\calG_{a_u}: H_0^1(\Omega)\rightarrow H_0^1(\Omega)$ satisfying
	\begin{equation*}
		\left(z, \calG_{a_u} w\right)_{a_u(\Omega)}  = (z,w)_{L^2(\Omega)},\quad\forall z\in H_0^1(\Omega).
	\end{equation*}
	Then it is straightforward to show that the $a_u$-gradient of the energy $E$ is given by 
	\begin{equation*}
		\nabla_{a_u}  E(u)= u.
	\end{equation*}
	Similar to Lemma~\ref{lem:proj_H1}, for $u\in \calM$ and $\xi\in H_0^1(\Omega)$, the projection of $\xi$ onto $\calT_u\calM$, with respect to the $a_u$-inner product, is
	\begin{equation*}
		\calP_{\calT_u\calM, a_u}(\xi) = \xi - \frac{(\xi, u)_{L^2(\Omega)}}{\norm{\calG_{a_u} u}_{a_u(\Omega)}^2}\calG_{a_u} u.
	\end{equation*}
	Therefore, the corresponding Riemannian gradient with respect to the $a_u$-inner product at $u\in\calM$ can be computed as:
	\begin{equation*}
		\nabla^{\mathcal{R}}_{a_u}E(u) = \calP_{\calT_u\calM, a_u}(\nabla_{a_u}E(u)) = u - \frac{1}{\norm{\calG_{a_u} u }_{a_u(\Omega)}^2 }\calG_{a_u} u,
	\end{equation*}
	which leads to the projected $a_u$-gradient descent scheme
	\begin{equation}\label{au_scheme}
		u_{n+1} = R\left(u_n - \alpha\ \nabla^{\mathcal{R}}_{a_{u_n}} E(u)\right),\quad n= 1,2\cdots.
	\end{equation}
	
	\section{Main Results}\label{sec:main}
	
	In this section, we state our main results on the convergence of the $H^1$-scheme, the $a_0$-scheme, and the $a_u$-scheme. All convergent results we establish are in sense of strong convergence with respect to the $H^1_0(\Omega)$ norm.
	
	\subsection{Global convergence} 
	
	It is proved in \cite{henning2020sobolev} that for $a_u$-scheme \eqref{au_scheme} with proper stepsizes, the energy decays along the iterations, i.e., $E(u_{n+1})\leq E(u_n)$, and the sequence $\{u_n\}_{n=0}^\infty$ has a subsequence that converges in $H_0^1(\Omega)$ to a critical point of \eqref{min_energy_problem}. However, as mentioned in \cite{henning2020sobolev}, it was an open question whether the iterates $\{u_n\}_{n=0}^\infty$ generated by the $H^1$-scheme \eqref{H1_scheme} or the $a_0$-scheme \eqref{a0_scheme} converge to critical points with decaying energy. 
	We give an affirmative answer to this question in Theorems~\ref{thm:energy_decay_H1}, \ref{thm:energy_decay_a0}, \ref{thm:global_conv_H1}, and \ref{thm:global_conv_a0}.
	
	\begin{theorem}[Energy decay for $H^1$-scheme]\label{thm:energy_decay_H1}
		Suppose that $d\in\{1,2,3\}$ and that $u_0\in \calM\subset H_0^1(\Omega)$. Let $\{u_n\}_{n=0}^\infty\subset \calM$ be the iterates generated by the $H^1$-scheme \eqref{H1_scheme} starting at $u_0$. There exist constants $C_u$, $C_g$, and $C_\alpha\leq 1$ depending only on $\Omega$, $d$, $\beta$, $V$, and $\norm{u_0}_{H_0^1(\Omega)}$, such that as long as the step size satisfies $0<\alpha_{\min} \leq \alpha_n\leq \alpha_{\max}\leq C_\alpha,\ \forall~n\geq 0$, the followings hold for any $n\geq 0$:
		\begin{itemize}
			\item[(i)] $\norm{u_n}_{H_0^1(\Omega)}\leq C_u$.
			\item[(ii)] $\norm{\nabla^{\mathcal{R}}_{H^1}E(u_n)}_{H_0^1(\Omega)}\leq \norm{\nabla_{H^1} E(u_n)}_{H_0^1(\Omega)} \leq C_g$.
			\item[(iii)] $E(u_n) - E(u_{n+1})\geq \frac{\alpha_{\min}}{2} \norm{\nabla^{\mathcal{R}}_{H^1} E(u_n)}_{H_0^1(\Omega)}^2$.
		\end{itemize}
	\end{theorem}
	
	\begin{theorem}[Energy decay for $a_0$-scheme]\label{thm:energy_decay_a0}
		Suppose that $d\in\{1,2,3\}$ and that $u_0\in \calM\subset H_0^1(\Omega)$. Let $\{u_n\}_{n=0}^\infty\subset \calM$ be the iterates generated by the $a_0$-scheme \eqref{a0_scheme} starting at $u_0$. There exist constants $C_u$, $C_g$, and $C_\alpha\leq 1$ depending only on $\Omega$, $d$, $\beta$, $V$, and $\norm{u_0}_{H_0^1(\Omega)}$, such that as long as the step size satisfies $0<\alpha_{\min} \leq \alpha_n\leq \alpha_{\max}\leq C_\alpha,\ \forall~n\geq 0$, the followings hold for any $n\geq 0$:
		\begin{itemize}
			\item[(i)] $\norm{u_n}_{a_0(\Omega)}\leq C_u$.
			\item[(ii)] $\norm{\nabla^{\mathcal{R}}_{a_0}E(u_n)}_{a_0(\Omega)}\leq \norm{\nabla_{a_0} E(u_n)}_{a_0(\Omega)} \leq C_g$.
			\item[(iii)] $E(u_n) - E(u_{n+1})\geq \frac{\alpha_{\min}}{2} \norm{\nabla^{\mathcal{R}}_{a_0} E(u_n)}_{a_0(\Omega)}^2$.
		\end{itemize}
	\end{theorem}
	
	The boundedness of $\{u_n\}$ as in Theorem~\ref{thm:energy_decay_H1}~(i) or Theorem~\ref{thm:energy_decay_a0}~(i) implies that the sequence $\{u_n\}$ has weak limits in $H_0^1(\Omega)$ ($\norm{\cdot}_{H_0^1(\Omega)}$ and $\norm{\cdot}_{a_0(\Omega)}$ are equivalent as in Lemma~\ref{lem:equiv_a0_H1}), and the following theorems show that any weak limit is a critical point of $E$.
	
	\begin{theorem}[Global convergence for $H^1$-scheme]\label{thm:global_conv_H1}
		Under the same assumptions of Theorem~\ref{thm:energy_decay_H1}, for any weak limit $u^*$ of $\{u_n\}_{n=0}^\infty$ in $H_0^1(\Omega)$, $u^*$ is a critical point of the problem \eqref{min_energy_problem}, i.e., $\nabla^{\mathcal{R}}_{H^1}E(u^*) = 0$, and $\{u_n\}_{n=0}^\infty$ has a subsequence that converges to $u^*$ strongly in $H_0^1(\Omega)$.
	\end{theorem}
	
	\begin{theorem}[Global convergence for $a_0$-scheme]\label{thm:global_conv_a0}
		Under the same assumptions of Theorem~\ref{thm:energy_decay_a0}, for any weak limit $u^*$ of $\{u_n\}_{n=0}^\infty$ in $H_0^1(\Omega)$, $u^*$ is a critical point of the problem \eqref{min_energy_problem}, i.e., $\nabla^{\mathcal{R}}_{a_0}E(u^*) = 0$, and $\{u_n\}_{n=0}^\infty$ has a subsequence that converges to $u^*$ strongly in $H_0^1(\Omega)$.
	\end{theorem}
	
	
	For the $a_u$-scheme, the energy decay property and the global convergence to critical points are established in \cite{henning2020sobolev}. In addition, \cite{henning2020sobolev} provides a stronger global convergence result for $a_u$-scheme: the whole sequence $\{u_n\}_{n=0}^\infty$ converges to the ground state, not just a critical point of the energy functional. The idea is that the $a_u$-scheme is positive preserving, i.e., $u_n\geq 0$ implies $u_{n+1}\geq 0$, and that the ground state is the unique positive eigenfunction of \eqref{eigen_problem}. However, the positive preserving property is not guaranteed for the $H^1$-scheme or the $a_0$-scheme, and hence the arguments in \cite{henning2020sobolev} for global convergence to ground state can not be applied directly  to the $H^1$-scheme or the $a_0$-scheme. 

	\subsection{Fast local convergence}
	
	In this section, we always denote $u^*\in H_0^1(\Omega)$ as the ground state of the eigenvalue problem \eqref{eigen_problem}. Consider the linearized problem at $u^*$:
	\begin{equation}\label{linear_eigen_problem}
		\begin{cases}
			\left(-\Delta + V + \beta |u^*|^2\right) u = \lambda u, & \text{in }\Omega,\\
			u = 0, &\text{on }\partial\Omega.
		\end{cases}
	\end{equation}
	By \cite{zhang2019exponential}*{Theorem 3.1}, $u^*$ is also the ground state of the linearized problem \eqref{linear_eigen_problem} and is H\"older continuous, if $\Omega$ is convex Lipschitz or $\partial \Omega$ is smooth. 
	
	\begin{assumption}[Positive eigengap for the linearized problem]\label{asp:eigengap}
		Let $\lambda_0$ and $\lambda_1$ be the smallest and the second smallest eigenvalue of the problem \eqref{linear_eigen_problem}, respectively. We assume that $\lambda_1 > \lambda_0$.
	\end{assumption}
	
	The existence of spectral gap of the linearized problem \ref{linear_eigen_problem} in Assumption \ref{asp:eigengap} can be verified using the Krein-Rutman theorem under very mild assumption on the potential $V$; see e.g. \cite{amann1976fixed}. With the assumption above, we are ready to state the local exponential convergence of both $H^1$-scheme and $a_u$-scheme. 
	
	\begin{theorem}[Local convergence for $H^1$-scheme]\label{thm:local_conv_H1}
		Suppose that assumptions made in Theorem~\ref{thm:energy_decay_H1} and Assumption~\ref{asp:eigengap} hold. In addition assume that the stepsizes $\alpha_n \leq \alpha_{\max}$ where $\alpha_{\max}$ satisfies that
		\begin{equation}\label{asp:alpha_local}
			1+ L_g^2 \alpha_{\max}^2 - \alpha_{\max} \min\left\{1,\frac{\lambda_1 - \lambda_0}{4\lambda_0}\right\}<1,
		\end{equation}
		where $L_g$ is a constant depending only on $\Omega$, $d$, $\beta$, $V$, $u^*$, and $\norm{u_0}_{H_0^1(\Omega)}$.
		Then the sequence $\{u_n\}_{n=0}^\infty\subset\calM$ generated by the $H^1$-scheme \eqref{H1_scheme} with initial condition $u_0$ close enough to $u^*$ converges  exponentially  in $H_0^1(\Omega)$ to the ground state $u^*$.
	\end{theorem}
	
	\begin{theorem}[Local convergence for $a_0$-scheme] \label{thm:local_conv_a0} 
		Suppose that assumptions made in Theorem~\ref{thm:energy_decay_a0} and Assumption~\ref{asp:eigengap} hold, and that $\alpha_{\max}$ satisfies \eqref{asp:alpha_local} for some constant $L_g$ depending only on $\Omega$, $d$, $\beta$, $V$, $u^*$, and $\norm{u_0}_{H_0^1(\Omega)}$. Then the sequence $\{u_n\}_{n=0}^\infty$ generated by the $a_0$-scheme \eqref{a0_scheme} with initial condition $u_0$ close enough to $u^*$ converges  exponentially  in $H_0^1(\Omega)$ to the ground state $u^*$.
	\end{theorem}
	
	\begin{theorem}[Local convergence for $a_u$-scheme]\label{thm:local_conv_au} 
		Suppose that Assumption~\ref{asp:eigengap} holds. There exists constant $C_\alpha$ depending on $\Omega$, $d$, $\beta$, $V$, and $\norm{u_0}_{H_0^1(\Omega)}$, and constant $L_g$ depending on $\Omega$, $d$, $\beta$, $V$, $u^*$, and $\norm{u_0}_{H_0^1(\Omega)}$, such that if $0<\alpha_{\min}\leq \alpha_n\leq \alpha_{\max}\leq C_\alpha$ with \eqref{asp:alpha_local},
		then the sequence $\{u_n\}_{n=0}^\infty$ generated by the $a_u$-scheme \eqref{au_scheme} with initial condition $u_0$ close enough to $u^*$ converges exponentially in $H_0^1(\Omega)$ to the ground state $u^*$.
	\end{theorem}
	
	Theorem \ref{thm:local_conv_H1} is to the best of our knowledge the first quantitative convergence result on the $H^1$-scheme (for previous qualitative convergence results, see \cites{kazemi2010minimizing,henning2020sobolev}). The result in Theorem \ref{thm:local_conv_au} recovers the same exponential convergence result of \cite{zhang2019exponential} but without making extra boundedness assumption on the iterates in $L^\infty(\Omega)$, which cannot be guaranteed. We remark that the spectral gap assumption is essential and necessary to obtain exponential convergence even for inverse iterations of linear eigenvalue problems. On the technical level, the spectral gap guarantees the locally strong convexity of the energy $E$ with respect to the $L^2$-norm; see Lemma \ref{lem:Elocalconvex}. We also refer to \cite{henning2022dependency} for more discussions on the role of spectral gap assumption in convergence of inverse iterations for nonlinear eigenvalue problems.
	
	Let us remark that all three schemes we analyze in this work are semi-discretized schemes without spatial discretization. The analysis of fully discretized schemes would be an interesting future research direction. 
	
	\section{Proof of Global Convergence}\label{sec:proofglobal}
	
	We prove Theorem~\ref{thm:energy_decay_H1} and Theorem~\ref{thm:global_conv_H1} in this section and omit the proofs of Theorem~\ref{thm:energy_decay_a0} and Theorem~\ref{thm:global_conv_a0} that follow almost the same lines, with the difference that all $H_0^1$-inner products and norms are replaced by $a_0$-inner products and norms. Note that $\norm{\cdot}_{a_0(\Omega)}$ is equivalent to $\norm{\cdot}_{H_0^1(\Omega)}$ (see Lemma~\ref{lem:equiv_a0_H1}).
	
	\subsection{Technical lemmas}
	We will frequently use the following Sobolev embeddings that hold for $d< 4$:
	\begin{align}
		& \norm{u}_{L^4(\Omega)}\leq C_1\norm{u}_{H_0^1(\Omega)},\quad\forall~u\in H_0^1(\Omega), \label{embed_L4}\\
		& \norm{u}_{H^{-1}(\Omega)}\leq C_2\norm{u}_{L^{4/3}(\Omega)},\quad \forall~u\in L^{4/3}(\Omega) \label{embedL43}.
	\end{align}
	In addition, the Poincaré inequality holds for some $C_3>0$:
	\begin{equation}\label{poincare_ineq}
		\norm{u}_{L^2(\Omega)}\leq C_3 \norm{u}_{H_0^1(\Omega)},\quad \forall~u\in H_0^1(\Omega).
	\end{equation}
	Note that the constants $C_1$, $C_2$, and $C_3$ only depend on $\Omega$ and $d$, and that the embeddings $H_0^1(\Omega)\subset\subset L^4(\Omega)$ and $H_0^1(\Omega)\subset\subset L^2(\Omega)$ are both compact.
	
	\begin{lemma}\label{lem:Gu}
		For any $u\in H_0^1(\Omega)$, the followings hold:
		\begin{itemize}
			\item[(i)] $\norm{\calG_{H^1} u}_{H_0^1(\Omega)}\leq \norm{u}_{H^{-1}(\Omega)}$;
			\item[(ii)] $\norm{\calG_{H^1} u}_{H_0^1(\Omega)}\leq C_3 \norm{u}_{L^2(\Omega)}$.
		\end{itemize}
	\end{lemma}
	
	\begin{proof}
		Let $g=\calG_{H^1} u$. Then
		\begin{equation*}
			\norm{g}^2_{H_0^1(\Omega)} = (g, \calG_{H^1} u)_{H_0^1(\Omega)} = (g, u)_{L^2(\Omega)}\leq \norm{g}_{H_0^1(\Omega)}\norm{u}_{H^{-1}(\Omega)},
		\end{equation*}
		which implies $\norm{\calG_{H^1} u}_{H_0^1(\Omega)}\leq \norm{u}_{H^{-1}(\Omega)}$. Combining the above equation with the Poincar\'e inequality \eqref{poincare_ineq}, one has
		\begin{equation*}
			\norm{g}_{L^2(\Omega)}\norm{g}_{H_0^1(\Omega)}\leq C_3 \norm{g}^2_{H_0^1(\Omega)} = C_3 (g, u)_{L^2(\Omega)}\leq C_3 \norm{g}_{L^2(\Omega)}\norm{u}_{L^2(\Omega)},
		\end{equation*}
		for some constant $C_3$. This implies $\norm{\calG_{H^1} u}_{H_0^1(\Omega)}\leq C_3 \norm{u}_{L^2(\Omega)}$. 
	\end{proof}
	
	\begin{lemma}\label{lem:esti_gradEu}
		For any $u\in\calM$, it holds that
		\begin{equation}\label{eq:grad_esti}
			\begin{split}
				\norm{\nabla^{\mathcal{R}}_{H^1}E(u)}_{H_0^1(\Omega)}& \leq \norm{\nabla_{H^1} E(u)}_{H_0^1(\Omega)} \\
				&\leq \norm{u}_{H_0^1(\Omega)} + C_3 V_{\max} + \beta  C_1^3 C_2\norm{u}_{H_0^1(\Omega)}^3.
			\end{split}
		\end{equation}
	\end{lemma}
	
	\begin{proof}
		Note that
		\begin{equation*}
			\nabla_{H^1} E(u) = \nabla^{\mathcal{R}}_{H^1} E(u) + \frac{(\nabla_{H^1} E(u), u)_{L^2(\Omega)}}{\norm{\calG_{H^1} u}_{H_0^1(\Omega)}^2} \calG_{H^1} u,
		\end{equation*}
		and that
		\begin{multline*}
			\left(\nabla^{\mathcal{R}}_{H^1} E(u), \frac{(\nabla_{H^1} E(u), u)_{L^2(\Omega)}}{\norm{\calG_{H^1} u}_{H_0^1(\Omega)}^2} \calG_{H^1} u\right)_{H_0^1(\Omega)}  \\
			= \frac{(\nabla_{H^1} E(u), u)_{L^2(\Omega)}}{\norm{\calG_{H^1} u}_{H_0^1(\Omega)}^2}\left(\nabla^{\mathcal{R}}_{H^1} E(u),  u\right)_{L^2(\Omega)}= 0.
		\end{multline*}
		So it holds that
		\begin{equation*}
			\norm{\nabla_{H^1} E(u)}_{H_0^1(\Omega)}^2 = \norm{\nabla^{\mathcal{R}}_{H^1} E(u)}_{H_0^1(\Omega)}^2 + \norm{\frac{(\nabla_{H^1} E(u), u)_{L^2(\Omega)}}{\norm{\calG_{H^1} u}_{H_0^1(\Omega)}^2} \calG_{H^1} u}_{H_0^1(\Omega)}^2,
		\end{equation*}
		which implies
		\begin{equation*}
			\norm{\nabla^{\mathcal{R}}_{H^1}E(u)}_{H_0^1(\Omega)}\leq \norm{\nabla_{H^1} E(u)}_{H_0^1(\Omega)}.
		\end{equation*}
		One also has that
		\begin{equation}\label{eq:gradEu}
			\begin{split}
				\norm{\nabla_{H^1} E(u)}_{H_0^1(\Omega)}& \leq \norm{u}_{H_0^1(\Omega)} + \norm{\calG_{H^1} (V u)}_{H_0^1(\Omega)}  + \beta \norm{\calG_{H^1} (|u|^2 u)}_{H_0^1(\Omega)}\\
				& \leq \norm{u}_{H_0^1(\Omega)} + C_3 \norm{V u}_{L^2(\Omega)} + \beta \norm{u^3}_{H^{-1}(\Omega)}\\
				& \leq \norm{u}_{H_0^1(\Omega)} + C_3 V_{\max} +  \beta C_2 \norm{u^3}_{L^{4/3}(\Omega)}\\
				& = \norm{u}_{H_0^1(\Omega)} + C_3 V_{\max} + \beta C_2 \norm{u}_{L^{4}(\Omega)}^3\\
				& \leq \norm{u}_{H_0^1(\Omega)} + C_3 V_{\max} + \beta C_1^3 C_2 \norm{u}_{H_0^1(\Omega)}^3,
			\end{split}
		\end{equation}
		where we used Lemma~\ref{lem:Gu}. This proves \eqref{eq:grad_esti}.
	\end{proof}
	
	\begin{lemma}\label{lem:esti_retraction}
		For any $u\in\calM$ and $\xi\in \calT_u\calM$, it holds that
		\begin{equation}\label{esti_retraction}
			\norm{R(u+\xi) - (u+\xi)}_{H_0^1(\Omega)}\leq \frac{1}{2}\norm{\xi}^2_{L^2(\Omega)}\norm{u+\xi}_{H_0^1(\Omega)}.
		\end{equation}
	\end{lemma}
	
	\begin{proof}
		It follows from $(u,\xi)_{L^2(\Omega)}=0$ that $\norm{u+\xi}_{L^2(\Omega)}^2 = \norm{u}_{L^2(\Omega)}^2 + \norm{\xi}_{L^2(\Omega)}^2 = 1 + \norm{\xi}_{L^2(\Omega)}^2$, which implies that
		\begin{equation*}
			R(u+\xi) - (u+\xi) = \left(\frac{1}{\norm{u+\xi}_{L^2(\Omega)}}-1\right)(u+\xi) = \left(\left( 1 + \norm{\xi}_{L^2(\Omega)}^2\right)^{-\frac{1}{2}} - 1\right)(u+\xi).
		\end{equation*}
		Set $f(x) = (1+x)^{-1/2}$. Then $f'(x) = -\frac{1}{2}(1+x)^{-3/2}\geq - \frac{1}{2}$ for all $x\geq 0$. Therefore, one can obtain that
		\begin{multline*}
			0 \geq \left( 1 + \norm{\xi}_{L^2(\Omega)}^2\right)^{-\frac{1}{2}} - 1  = f\left( \norm{\xi}_{L^2(\Omega)}^2\right) - f(0)\\
			=  \int_0^{\norm{\xi}_{L^2(\Omega)}^2} f'(x)dx \geq \int_0^{\norm{\xi}_{L^2(\Omega)}^2} \left(-\frac{1}{2}\right)dx = -\frac{1}{2}\norm{\xi}_{L^2(\Omega)}^2.
		\end{multline*}
		\eqref{esti_retraction} then holds by combining the estimations above.
	\end{proof}
	
	\begin{lemma}\label{lem:linear_error}
		For any $u,v\in H^1_0(\Omega)$, it holds that
		\begin{multline*}
			\left|E(u+v)-E(u)-(\nabla_{H^1} E(u),v)_{H_0^1(\Omega)}\right|
			\leq \frac{1 +  C_3^2 V_{\max}}{2}\norm{v}_{H_0^1(\Omega)}^2\\ + \frac{3\beta C_1^4}{2}\norm{u}_{H_0^1(\Omega)}^2 \norm{v}_{H_0^1(\Omega)}^2 +\beta C_1^4 \norm{u}_{H_0^1(\Omega)}\norm{v}_{H_0^1(\Omega)}^3 + \frac{\beta C_1^4}{4} \norm{v}_{H_0^1(\Omega)}^4.
		\end{multline*}
	\end{lemma}
	
	\begin{proof}
		We have
		\begin{align*}
			& E(u+v) - E(u) \\
			= & \int_\Omega \left(\frac{1}{2}|\nabla u +\nabla v|^2 -\frac{1}{2}|\nabla u|^2\right) + \int_\Omega \left(\frac{1}{2} V|u+v|^2 - \frac{1}{2} V|u|^2\right) \\
			& \qquad\qquad +\int_\Omega \left(\frac{\beta}{4}|u+v|^4 - \frac{\beta}{4} |u|^4\right) \\
			= & \int_\Omega \left(\nabla u\cdot \nabla v + Vuv +\beta |u|^2 u v\right)\\
			&\qquad\qquad + \frac{1}{2}\int_\Omega \left(|\nabla v|^2 +  V|v|^2\right) +\int_\Omega \left(\frac{3\beta}{2} |u|^2 |v|^2 + \beta |v|^2 u v + \frac{\beta}{4} |v|^4 \right)\\
			=& \left(\nabla_{H^1} E(u),v\right)_{H_0^1(\Omega)}\\
			&\qquad\qquad + \frac{1}{2}\int_\Omega \left(|\nabla v|^2 +  V|v|^2\right) +\int_\Omega \left(\frac{3\beta}{2} |u|^2 |v|^2 + \beta |v|^2 u v + \frac{\beta}{4} |v|^4 \right),
		\end{align*}
		which leads to
		\begin{align*}
			&\left|E(u+v)-E(u) - \left(\nabla_{H^1} E(u),v\right)_{H_0^1(\Omega)}\right|\\
			\leq &\frac{1}{2}\int_\Omega \left(|\nabla v|^2 +  V|v|^2\right)+ \int_\Omega \left(\frac{3\beta}{2} |u|^2 |v|^2 + \beta |u| |v|^3 + \frac{\beta}{4} |v|^4\right) \\
			\leq &\frac{1 + C_3^2 V_{\max}}{2}\norm{v}_{H_0^1(\Omega)}^2+ \frac{3\beta}{2}\norm{u}_{L^4(\Omega)}^2 \norm{v}_{L^4(\Omega)}^2 \\
			&\qquad\qquad +\beta \norm{u}_{L^4(\Omega)} \norm{v}_{L^4(\Omega)}^3 + \frac{\beta}{4} \norm{v}_{L^4(\Omega)}^4 \\
			\leq& \frac{1 + C_3^2 V_{\max}}{2}\norm{v}_{H_0^1(\Omega)}^2 + \frac{3\beta C_1^4}{2}\norm{u}_{H_0^1(\Omega)}^2 \norm{v}_{H_0^1(\Omega)}^2 \\
			&\qquad \qquad +\beta C_1^4 \norm{u}_{H_0^1(\Omega)} \norm{v}_{H_0^1(\Omega)}^3 + \frac{\beta C_1^4}{4} \norm{v}_{H_0^1(\Omega)}^4.
		\end{align*}
	\end{proof}

	\subsection{Proofs of Theorem~\ref{thm:energy_decay_H1} and Theorem~\ref{thm:global_conv_H1}}

	\begin{proof}[Proof of Theorem~\ref{thm:energy_decay_H1}]
		Set
		\begin{align*}
			& C_u = \left((1 +C_3^2 V_{\max}) \norm{u_0}_{H_0^1(\Omega)}^2 + \frac{\beta C_1^4}{2}\norm{u_0}_{H_0^1(\Omega)}^4\right)^{1/2}, \\
			& C_g = C_u + C_3 V_{\max} + \beta C_1^3 C_2 C_u^3,
		\end{align*}
		and we will determine the constant $C_\alpha\leq 1$ later. Our goal is to prove the three conclusions in Theorem~\ref{thm:energy_decay_H1}, i.e., 
		\begin{itemize}
			\item[(i)] $\norm{u_n}_{H_0^1(\Omega)}\leq C_u, \ \forall~n\geq 0$.
			\item[(ii)] $\norm{\nabla^{\mathcal{R}}_{H^1}E(u_n)}_{H_0^1(\Omega)}\leq \norm{\nabla_{H^1} E(u_n)}_{H_0^1(\Omega)} \leq C_g, \ \forall~n\geq 0$.
			\item[(iii)] $E(u_n) - E(u_{n+1})\geq \frac{\alpha_{\min}}{2} \norm{\nabla^{\mathcal{R}}_{H^1} E(u_n)}_{H_0^1(\Omega)}^2, \ \forall~n\geq 0$.
		\end{itemize}
		We prove (i), (ii), and (iii) by induction. It is clear that (i) holds for $n=0$. Suppose that (i) holds for $0,1,\dots,n$ and that (ii) and (iii) hold for $0,1,\dots,n-1$. We aim to show that (ii) and (iii) hold for $n$ and that (i) holds for $n+1$.
		
		It follows directly from Lemma~\ref{lem:esti_gradEu} and (i) that (ii) holds for $n$. We focus on (iii) then. Denote
		\begin{equation*}
			g_n = \nabla^{\mathcal{R}}_{H^1} E(u_n),\text{ and }\Tilde{u}_n = u_n - \alpha_n \nabla^{\mathcal{R}}_{H^1} E(u_n).
		\end{equation*} The iterative scheme is
		\begin{equation*}
			u_{n+1} = R(\Tilde{u}_n) =  R\left(u_n - \alpha_n g_n\right) = u_n - \alpha_n g_n + R_n = \Tilde{u}_n + R_n,
		\end{equation*}
		where
		\begin{equation}\label{def:Rn}
			R_n = R\left(u_n - \alpha_n g_n\right) - \left(u_n - \alpha_n g_n\right).
		\end{equation}
		According to Lemma~\ref{lem:esti_retraction}, it holds that
		\begin{equation}\label{eq:Rn}
			\begin{split}
				\norm{R_n}_{H_0^1(\Omega)} & \leq \frac{\alpha_n^2}{2}\norm{g_n}_{L^2(\Omega)}^2 \norm{u_n - \alpha_n g_n}_{H_0^1(\Omega)} \\
				& \leq \frac{\alpha_n^2}{2} (C_u + C_g) C_3^2 \norm{g_n}_{H_0^1(\Omega)}^2 \leq \frac{ (C_u + C_g) C_3^2 C_g^2}{2},
			\end{split}
		\end{equation}
		where we used $\alpha_n\leq C_\alpha\leq 1$. Similar to \eqref{eq:gradEu} $\|\nabla_{H^1}E(\tilde u_n)\|_{H_0^1(\Omega)}$ can be upper bounded as
		\begin{align*}
			\norm{\nabla_{H^1} E(\Tilde{u}_n)}_{H_0^1(\Omega)}& \leq \norm{\Tilde{u}_n}_{H_0^1(\Omega)} + C_3 V_{\max} \norm{\Tilde{u}_n}_{L^2(\Omega)} + \beta C_1^3 C_2 \norm{\Tilde{u}_n}_{H_0^1(\Omega)}^3 \\
			&\leq (1 + C_3^2 V_{\max}) \norm{\Tilde{u}_n}_{H_0^1(\Omega)}  + \beta C_1^3 C_2 \norm{\Tilde{u}_n}_{H_0^1(\Omega)}^3 \\
			& \leq (1 + C_3^2 V_{\max}) (C_u + C_g) + \beta C_1^3 C_2 (C_u + C_g)^3,
		\end{align*}
		where we used $\norm{\Tilde{u}_n}_{H_0^1(\Omega)} \leq \norm{u_n}_{H_0^1(\Omega)} + \alpha_n \norm{g_n}_{H_0^1(\Omega)} \leq C_u + C_g$. Then using Lemma~\ref{lem:linear_error}, one can estimate that
		\begin{align*}
			&|E(\Tilde{u}_n) - E(\Tilde{u}_n + R_n)|\\
			\leq & \left|(\nabla_{H^1} E(\Tilde{u}_n), R_n)_{H_0^1(\Omega)} \right| + \frac{1 + C_3^2 V_{\max}}{2}\norm{R_n}_{H_0^1(\Omega)}^2+ \frac{3\beta C_1^4}{2}\norm{\Tilde{u}_n}_{H_0^1(\Omega)}^2 \norm{R_n}_{H_0^1(\Omega)}^2\\
			&\quad  +\beta C_1^4 \norm{\Tilde{u}_n}_{H_0^1(\Omega)} \norm{R_n}_{H_0^1(\Omega)}^3 + \frac{\beta C_1^4}{4} \norm{R_n}_{H_0^1(\Omega)}^4\\
			\leq & \norm{R_n}_{H_0^1(\Omega)}\left(\norm{\nabla_{H^1} E(\Tilde{u}_n)}_{H_0^1(\Omega)} + \frac{1 + C_3^2 V_{\max}}{2}\norm{R_n}_{H_0^1(\Omega)}\right.\\
			&\quad \left.+ \frac{3\beta C_1^4}{2}\norm{\Tilde{u}_n}_{H_0^1(\Omega)}^2 \norm{R_n}_{H_0^1(\Omega)} +\beta C_1^4 \norm{\Tilde{u}_n}_{H_0^1(\Omega)} \norm{R_n}_{H_0^1(\Omega)}^2 + \frac{\beta C_1^4}{4} \norm{R_n}_{H_0^1(\Omega)}^3\right) \\
			\leq & \alpha_n^2 C_R \norm{g_n}_{H_0^1(\Omega)}^2,
		\end{align*}
		where $C_R$ is a constant depending only on $\beta$, $C_u$, $C_g$, $V_{\max}$, $C_1$, $C_2$, and $C_3$, and that
		\begin{align*}
			&\left|E(u_n-\alpha_n g_n)-E(u_n)-(\nabla_{H^1} E(u_n),-\alpha_n g_n)_{H_0^1(\Omega)}\right|\\
			\leq& \frac{1 + C_3^2 V_{\max}}{2}\norm{\alpha_n g_n}_{H_0^1(\Omega)}^2+ \frac{3\beta C_1^4}{2}\norm{u_n}_{H_0^1(\Omega)}^2 \norm{\alpha_n g_n}_{H_0^1(\Omega)}^2 \\
			&\qquad\qquad +\beta C_1^4 \norm{u_n}_{H_0^1(\Omega)}\norm{\alpha_n g_n}_{H_0^1(\Omega)}^3 + \frac{\beta C_1^4}{4} \norm{\alpha_n g_n}_{H_0^1(\Omega)}^4 \\
			\leq & \alpha_n^2 \norm{g_n}_{H_0^1(\Omega)}^2\left(\frac{1 + C_3^2 V_{\max}}{2} +  \frac{3\beta C_1^4 C_u^2}{2} +\beta C_1^4 C_u C_g + \frac{\beta C_1^4 C_g^2}{4} \right).
		\end{align*}
		Combining the above two inequalities, one has
		\begin{align*}
			& E(u_n) - E(u_{n+1}) \\
			= & E(u_n) - E(\Tilde{u}_n) +E(\Tilde{u}_n) - E(\Tilde{u}_n + R_n) \\
			\geq &\alpha_n (\nabla_{H^1} E(u_n), g_n)_{H_0^1(\Omega)}\\ 
			&\qquad - \left|E(u_n-\alpha_n g_n)-E(u_n)-(\nabla_{H^1} E(u_n),-\alpha_n g_n)_{H_0^1(\Omega)}\right| \\
			&\qquad - |E(\Tilde{u}_n) - E(\Tilde{u}_n + R_n)| \\ 
			\geq & \alpha_n \norm{g_n}_{H_0^1(\Omega)}^2 \\
			&\qquad -\alpha_n^2 \norm{g_n}_{H_0^1(\Omega)}^2\left(\frac{1+C_3^2 V_{\max}}{2} +  \frac{3\beta C_1^4 C_u^2}{2} +\beta C_1^4 C_u C_g + \frac{\beta C_1^4 C_g^2}{4} \right) \\
			&\qquad - \alpha_n^2 C_R \norm{g_n}_{H_0^1(\Omega)}^2\\
			\geq & \frac{\alpha_{\min}}{2}\norm{g_n}_{H_0^1(\Omega)}^2,
		\end{align*}
		if
		\begin{equation*}
			\alpha_{\max}\left(C_R + \frac{1 + C_3^2 V_{\max}}{2} +  \frac{3\beta C_1^4 C_u^2}{2} +\beta C_1^4 C_u C_g + \frac{\beta C_1^4 C_g^2}{4} \right)\leq \frac{1}{2},
		\end{equation*}
		which can be guaranteed by $\alpha_{\max}\leq C_\alpha$, where $C_\alpha$ is a sufficiently small constant. This means that (iii) holds for $n$. Then we have that
		\begin{align*}
			\norm{u_{n+1}}_{H_0^1(\Omega)}^2 &\leq 2 E(u_{n+1})\leq 2E(u_n) \leq \cdots\leq 2E(u_0) \\
			& \leq \norm{u_0}_{H_0^1(\Omega)}^2 + V_{\max} \norm{u_0}_{L^2(\Omega)}^2 + \frac{\beta}{2}\norm{u_0}_{L^4(\Omega)}^4\\
			& \leq (1 + C_3^2 V_{\max}) \norm{u_0}_{H_0^1(\Omega)}^2 + \frac{\beta C_1^4}{2}\norm{u_0}_{H_0^1(\Omega)}^4,
		\end{align*}
		which implies that $\norm{u_{n+1}}_{H_0^1(\Omega)} \leq  C_u$, i.e., (i) holds for $n+1$. The proof is hence completed.
	\end{proof}
	
	We need the following theorem for proving Theorem~\ref{thm:global_conv_H1}. Some ideas in the proof of Theorem~\ref{thm:strong_conv} follow the proof of Theorem 4.9 in \cite{henning2020sobolev}.
	
	\begin{theorem}\label{thm:strong_conv}
		Suppose that $\{v_n\}_{n=0}^\infty$ is a bounded sequence in $\calM\subset H_0^1(\Omega)$ with
		\begin{equation*}
			\lim_{n\rightarrow\infty}  \norm{\nabla^{\mathcal{R}}_{H^1} E(v_n)}_{H_0^1(\Omega)} = 0.
		\end{equation*}
		Let $v^*$ be any weak limit of $\{v_n\}_{n=0}^\infty$ in $H_0^1(\Omega)$. Then $v^*$ is a critical point of $E$ and $\{v_n\}_{n=0}^\infty$ has a subsequence that converges to $v^*$ strongly in $H_0^1(\Omega)$
	\end{theorem}
	
	\begin{proof}
		By compact embeddings $H_0^1(\Omega)\subset\subset L^2(\Omega)$ and $H_0^1(\Omega)\subset\subset L^4(\Omega)$, there exists a subsequence of $\{v_n\}_{n=0}^\infty$ converging weakly in $H_0^1(\Omega)$ and strongly in $L^2(\Omega)$ and $L^4(\Omega)$ to some $v^*\in H_0^1(\Omega)$. Without loss of generality, we assume that
		\begin{equation*}
			v_n\rightarrow v^*,\quad\text{weakly in } H_0^1(\Omega)\text{ and strongly in }L^2(\Omega)\text{ and }L^4(\Omega).
		\end{equation*}
		Notice that $\norm{\calG_{H^1}(v_n - v^*)}_{H_0^1(\Omega)}\leq C_3 \norm{v_n - v^*}_{L^2(\Omega)}$ and that $\norm{\calG_{H^1}(V v_n - V v^*)}_{H_0^1(\Omega)}\leq C_3\norm{ V  v_n - V v^*}_{L^2(\Omega)} \leq C_3 V_{\max} \norm{v_n - v^*}_{L^2(\Omega)}$ by Lemma~\ref{lem:Gu}~(ii). We can obtain that 
		\begin{equation*}
			\calG_{H^1} v_n\rightarrow \calG_{H^1} v^*, \quad \text{strongly in }H_0^1(\Omega)\text{ and }L^2(\Omega),
		\end{equation*}
		and that
		\begin{equation*}
			\calG_{H^1} (V v_n)\rightarrow \calG_{H^1} (V v^*), \quad \text{strongly in }H_0^1(\Omega)\text{ and }L^2(\Omega),
		\end{equation*}
		Denote $e_n = v_n - v^*$. Then $\lim\limits_{n\rightarrow\infty}\norm{e_n}_{L^4(\Omega)} = 0$. Note that
		\begin{align*}
			& \norm{v_n^3 - (v^*)^3}_{H^{-1}(\Omega)} =  \norm{(v^*+e_n)^3 - (v^*)^3}_{H^{-1}(\Omega)} \\
			\leq & 3 \norm{(v^*)^2 e_n}_{H^{-1}(\Omega)}+3 \norm{v^* e_n^2}_{H^{-1}(\Omega)}+ \norm{ e_n^3}_{H^{-1}(\Omega)}\\
			\leq & 3 C_2 \norm{(v^*)^2 e_n}_{L^{4/3}(\Omega)}+3C_2 \norm{v^* e_n^2}_{L^{4/3}(\Omega)}+ C_2 \norm{ e_n^3}_{L^{4/3}(\Omega)}\\
			= & 3 C_2 \left(\int_\Omega (v^*)^{8/3} e_n^{4/3}\right)^{3/4} + 3 C_2 \left(\int_\Omega (v^*)^{4/3} e_n^{8/3}\right)^{3/4} + C_2 \left(\int_\Omega  e_n^4\right)^{3/4} \\
			\leq & 3 C_2\norm{v^*}_{L^4(\Omega)}^2 \norm{e_n}_{L^4(\Omega)} +3 C_2\norm{v^*}_{L^4(\Omega)} \norm{e_n}_{L^4(\Omega)}^2 + C_2\norm{e_n}_{L^4(\Omega)}^3\\
			\rightarrow & 0,
		\end{align*}
		where we used \eqref{embedL43}. By Lemma~\ref{lem:Gu}~(i), we have $\norm{\calG_{H^1}(v_n^3 - (v^*)^3)}_{H_0^1(\Omega)}\leq \norm{v_n^3 - (v^*)^3}_{H^{-1}(\Omega)}$ thus 
		\begin{equation*}
			\calG_{H^1} (|v_n|^2 v_n)\rightarrow \calG_{H^1} (|v^*|^2 v^*), \quad \text{strongly in }H_0^1(\Omega)\text{ and }L^2(\Omega).
		\end{equation*}
		Then one has that
		\begin{equation*}
			\frac{1 + \left(\calG_{H^1} (V v_n + \beta |v_n|^2 v_n), v_n\right)_{L^2(\Omega)}}{\norm{\calG_{H^1} v_n}_{H_0^1(\Omega)}^2}\rightarrow \frac{1 + \left(\calG_{H^1} (V v^* + \beta |v^*|^2 v^*), v^*\right)_{L^2(\Omega)}}{\norm{\calG_{H^1} v^*}_{H_0^1(\Omega)}^2},
		\end{equation*}
		and hence that
		\begin{align*}
			&\nabla^{\mathcal{R}}_{H^1}E(v_n) \\
			=& v_n + \calG_{H^1}( V v_n +\beta |v_n|^2 v_n) - \frac{1 + \left(\calG_{H^1}( V v_n + \beta |v_n|^2 v_n), v_n\right)_{L^2(\Omega)}}{\norm{\calG_{H^1} v_n}_{H_0^1(\Omega)}^2} \calG_{H^1} v_n\\
			\rightarrow & v^* + \calG_{H^1}( V v^*+\beta |v^*|^2 v^*) - \frac{1 + \left(\calG_{H^1} (V v^* + \beta |v^*|^2 v^*), v^*\right)_{L^2(\Omega)}}{\norm{\calG_{H^1} v^*}_{H_0^1(\Omega)}^2} \calG_{H^1} v^*\\
			=& \nabla^{\mathcal{R}}_{H^1}E(v^*), \quad\text{weakly in }H_0^1(\Omega),
		\end{align*}
		which combined with $\lim\limits_{n\to\infty} \norm{\nabla^{\mathcal{R}}_{H^1}E(v_n)}_{H_0^1(\Omega)} = 0$ yields that $\nabla^{\mathcal{R}}_{H^1}E(v^*) = 0$, i.e., 
		\begin{equation*}
			-\Delta v^* +  V v^*+\beta |v^*|^2 v^* = \frac{1 + \left(\calG_{H^1} (V v^* + \beta |v^*|^2 v^*), v^*\right)_{L^2(\Omega)}}{\norm{\calG_{H^1} v^*}_{H_0^1(\Omega)}^2}  v^*,
		\end{equation*}
		which states that $v^*$ is a critical point as well as an eigenfunction.
		
		We then prove that $v_n$ converges strongly to $v^*$ in $H_0^1(\Omega)$. $\lim\limits_{n\to\infty}\norm{\nabla^{\mathcal{R}}_{H^1}E(v_n)}_{H_0^1(\Omega)} = 0$ and the boundedness of $\{v_n\}_{n=0}^\infty\subset H_0^1(\Omega)$ imply that
		\begin{equation*}
			(\nabla^{\mathcal{R}}_{H^1}E(v_n),v_n)_{H_0^1(\Omega)}\rightarrow 0 = (\nabla^{\mathcal{R}}_{H^1}E(v^*),v^*)_{H_0^1(\Omega)},
		\end{equation*}
		i.e.,
		\begin{align*}
			& \norm{v_n}_{H_0^1(\Omega)}^2 + ( V v_n, v_n)_{L^2(\Omega)}+ ( \beta |v_n|^2 v_n, v_n)_{L^2(\Omega)} \\
			&\qquad\qquad\qquad\qquad\qquad\qquad - \frac{1 + \left(\calG_{H^1} (V v_n+\beta |v_n|^2 v_n), v_n\right)_{L^2(\Omega)}}{\norm{\calG_{H^1} v_n}_{H_0^1(\Omega)}^2} \\
			\rightarrow & \norm{v^*}_{H_0^1(\Omega)}^2 + ( V v^*, v^*)_{L^2(\Omega)}+ ( \beta |v^*|^2 v^*, v^*)_{L^2(\Omega)} \\
			&\qquad\qquad\qquad\qquad\qquad\qquad - \frac{1 + \left(\calG_{H^1} (V v^*+\beta |v^*|^2 v^*), v^*\right)_{L^2(\Omega)}}{\norm{\calG_{H^1} v^*}_{H_0^1(\Omega)}^2}. 
		\end{align*} 
		Note that the second, third, and fourth terms above converge respectively. One has $\norm{v_n}_{H_0^1(\Omega)}\rightarrow \norm{v^*}_{H_0^1(\Omega)}$. Then the proof can be completed since the convergence of norm together with the weak convergence implies the strong convergence.
	\end{proof}
	
	\begin{proof}[Proof of Theorem~\ref{thm:global_conv_H1}]
		It follows from Theorem~\ref{thm:energy_decay_H1}~(iii) that
		\begin{equation*}
			\sum_{n=0}^\infty \norm{\nabla^{\mathcal{R}}_{H^1}E(u_n)}_{H_0^1(\Omega)}^2 \leq \frac{2}{\alpha_{\min}} \sum_{n=0}^\infty (E(u_n) - E(u_{n+1}))\leq \frac{2 E(u_0)}{\alpha_{\min}} < \infty,
		\end{equation*}
		which implies
		\begin{equation*}
			\lim_{n\rightarrow\infty}\norm{\nabla^{\mathcal{R}}_{H^1}E(u_n)}_{H_0^1(\Omega)} = 0.
		\end{equation*}
		Note that $\{u_n\}_{n=0}^\infty\subset\calM$ is bounded in $H_0^1(\Omega)$. We can obtain the desired results by applying Theorem~\ref{thm:strong_conv}.
	\end{proof}

	\section{Proof of Local Convergence}\label{sec:prooflocal}
	
	We present the proof of Theorem~\ref{thm:local_conv_H1} in this section. Theorem~\ref{thm:local_conv_a0} can be proved by following the same lines (see the remark at the beginning of Section~\ref{sec:proofglobal}) and its proof is hence omitted. Theorem~\ref{thm:local_conv_au} can also be proved with a similar analysis with some different technical lemmas, which is presented in Appendix~\ref{sec:pf_local_au}.
	
	\subsection{Technical lemmas}
	
	\begin{lemma}\label{lem:Elocalconvex}
		Suppose that Assumption~\ref{asp:eigengap} holds. Then for any $u\in \calM\subset H_0^1(\Omega)$, it holds that
		\begin{equation*}
			E(u) - E(u^*) \geq \frac{\lambda_0 -\lambda_1}{4}\norm{u - u^*}_{L^2(\Omega)}^2,
		\end{equation*}
		as long as $\norm{u - u^*}_{L^2(\Omega)}^2\leq 2$.
	\end{lemma}
	
	\begin{proof}
		We have
		\begin{align*}
			& E(u) - E(u^*) \\
			= & \int_\Omega \left(\frac{1}{2}|\nabla u|^2 + \frac{1}{2}V |u|^2 + \frac{\beta}{4} |u|^4 \right) - \int_\Omega \left( \frac{1}{2}|\nabla u^*|^2 + \frac{1}{2}V |u^*|^2 + \frac{\beta}{4} |u^*|^4 \right) \\
			= & \int_\Omega \left(\frac{1}{2}|\nabla u|^2 + \frac{1}{2}V |u|^2 + \frac{\beta}{2} |u^*|^2 |u|^2\right)  - \int_\Omega \left(\frac{1}{2}|\nabla u^*|^2 + \frac{1}{2}V |u^*|^2 + \frac{\beta}{2} |u^*|^2 |u^*|^2\right) \\
			&\qquad\qquad + \int_\Omega \left(\frac{\beta}{4} |u|^4 - \frac{\beta}{2} |u^*|^2 |u|^2 + \frac{\beta}{4}|u^*|^4\right)\\
			\geq & \int_\Omega \left(\frac{1}{2}|\nabla u|^2 + \frac{1}{2}V |u|^2 + \frac{\beta}{2} |u^*|^2 |u|^2\right)  - \int_\Omega \left(\frac{1}{2}|\nabla u^*|^2 + \frac{1}{2}V |u^*|^2 + \frac{\beta}{2} |u^*|^2 |u^*|^2\right).
		\end{align*}
		Let $u_{\parallel} = (u, u^*)_{L^2(\Omega)} u^*$ be the $L^2$-orthogonal projection of $u$ onto the subspace spanned by $u^*$, and let $u_\perp = u - u_{\parallel}$. It follows from the orthogonality that
		\begin{equation*}
			\norm{u_\parallel}_{L^2(\Omega)}^2 + \norm{u_\perp }_{L^2(\Omega)}^2 = \norm{u}_{L^2(\Omega)}^2 = 1.
		\end{equation*}
		Therefore,
		\begin{align*}
			& \int_\Omega \left(\frac{1}{2}|\nabla u|^2 + \frac{1}{2}V |u|^2 + \frac{\beta}{2} |u^*|^2 |u|^2\right) - \int_\Omega \left(\frac{1}{2}|\nabla u^*|^2 + \frac{1}{2}V |u^*|^2 + \frac{\beta}{2} |u^*|^2 |u^*|^2\right)\\
			= & \frac{1}{2}\left(u, (-\Delta  + V + \beta |u^*|^2) u\right)_{L^2(\Omega)} - \frac{1}{2}\left(u^*, (-\Delta  + V + \beta |u^*|^2) u^*\right)_{L^2(\Omega)} \\
			= & \frac{1}{2}\left(u_\parallel, (-\Delta  + V + \beta |u^*|^2) u_\parallel\right)_{L^2(\Omega)} + \frac{1}{2}\left(u_\perp, (-\Delta  + V + \beta |u^*|^2) u_\perp\right)_{L^2(\Omega)} - \frac{\lambda_0}{2}\\
			\geq & \frac{\lambda_0}{2} \norm{u_\parallel}_{L^2(\Omega)}^2 + \frac{\lambda_1}{2}\norm{u_\perp}_{L^2(\Omega)}^2 - \frac{\lambda_0}{2}\\
			= & \frac{\lambda_1 - \lambda_0}{2}\norm{u_\perp}_{L^2(\Omega)}^2.
		\end{align*}
		Notice also that
		\begin{align*}
			\norm{u_\perp}_{L^2(\Omega)}^2 = & 1 - \norm{u_\parallel}_{L^2(\Omega)}^2 = 1 - \left|(u, u^*)_{L^2(\Omega)}\right|^2\\
			= & 1 - \frac{1}{4}\left(\norm{u}_{L^2(\Omega)}^2 + \norm{u^*}_{L^2(\Omega)}^2 - \norm{u - u^*}_{L^2(\Omega)}^2\right)^2 \\
			= & 1-\frac{1}{4}\left(2 - \norm{u-u^*}_{L^2(\Omega)}^2\right)^2 \\
			= & \norm{u-u^*}_{L^2(\Omega)}^2  - \frac{1}{4}\norm{u-u^*}_{L^2(\Omega)}^4.
		\end{align*}
		So we have
		\begin{align*}
			E(u_n) - E(u^*) & \geq \frac{\lambda_1 - \lambda_0}{2} \norm{u-u^*}_{L^2(\Omega)}^2 - \frac{\lambda_1 - \lambda_0}{8}\norm{u-u^*}_{L^2(\Omega)}^4 \\ 
			& \geq \frac{\lambda_1 - \lambda_0}{4} \norm{u-u^*}_{L^2(\Omega)}^2,
		\end{align*}
		if $\norm{u - u^*}_{L^2(\Omega)}^2\leq 2$.
	\end{proof}

	\begin{lemma}\label{lem:Lg}
		Let $C_u$ be the constant as in Theorem~\ref{thm:energy_decay_H1}. Then there exists a constant $L_g$ depending only on $\Omega$, $d$, $\beta$, $V$, $u^*$ and $\norm{u_0}_{H_0^1(\Omega)}$ such that 
		\begin{equation*}
			\norm{\nabla^{\mathcal{R}}_{H^1}E(u)}_{H_0^1(\Omega)}\leq L_g\norm{u - u^*}_{H_0^1(\Omega)}
		\end{equation*}
		holds as long as $u\in \calM$  and  $u$ satisfies that $\norm{u}_{H_0^1(\Omega)}\leq C_u$ and $\norm{u - u^*}_{H_0^1(\Omega)}$ is sufficiently small.
	\end{lemma}
	
	\begin{proof}
		Denote
		\begin{equation*}
			\gamma = \frac{1 + \left(\calG_{H^1} (V u + \beta |u|^2 u), u\right)_{L^2(\Omega)}}{\norm{\calG_{H^1} u}_{H_0^1(\Omega)}^2},
		\end{equation*}
		and
		\begin{equation}\label{gammastar}
			\gamma^* = \frac{1 + \left(\calG_{H^1} (V u^* + \beta |u^*|^2 u^*), u^*\right)_{L^2(\Omega)}}{\norm{\calG_{H^1} u^*}_{H_0^1(\Omega)}^2}.
		\end{equation}
		It holds that
		\begin{align*}
			& \norm{\nabla^{\mathcal{R}}_{H^1}E(u)}_{H_0^1(\Omega)} =  \norm{\nabla^{\mathcal{R}}_{H^1}E(u) - \nabla^{\mathcal{R}}_{H^1} 
				E(u^*)}_{H_0^1(\Omega)}\\
			= & \left\|u + \calG_{H^1} (V u +\beta |u|^2 u) - \gamma \calG_{H^1} u  - u^* - \calG_{H^1} (V u^* + \beta |u^*|^2 u^*) -  \gamma^* \calG_{H^1} u^*\right\|_{H_0^1(\Omega)}\\
			\leq & \norm{u - u^*}_{H_0^1(\Omega)} + \norm{\calG_{H^1} (V u - V u^*)}_{H_0^1(\Omega)} + \beta \norm{\calG_{H^1} (u^3 - (u^*)^3)}_{H_0^1(\Omega)}\\
			&\qquad\qquad + |\gamma - \gamma^*|\cdot \norm{\calG_{H^1}u}_{H_0^1(\Omega)} + \gamma^* \norm{\calG_{H^1}(u - u^*)}_{H_0^1(\Omega)} \\
			\leq & \norm{u - u^*}_{H_0^1(\Omega)} + C_3 \norm{ V(u - u^*)}_{L^2(\Omega)} + \beta \norm{u^3 - (u^*)^3}_{H^{-1}(\Omega)}\\
			& \qquad\qquad + |\gamma - \gamma^*|\cdot C_3 \norm{u}_{L^2(\Omega)} + C_3 \gamma^* \norm{u - u^*}_{L^2(\Omega)}\\
			\leq & (1+ C_3^2 V_{\max} + C_3^2 \gamma^*)\norm{u-u^*}_{H_0^1(\Omega)} \beta \norm{u^3 - (u^*)^3}_{H^{-1}(\Omega)} + C_3 |\gamma - \gamma^*|,
		\end{align*}
		where we used Lemma~\ref{lem:Gu} and the Poincar\'e inequality \eqref{poincare_ineq}. The rest of the proof is to estimate $\norm{u^3 - (u^*)^3}_{H^{-1}(\Omega)}$ and $|\gamma - \gamma^*|$. We have
		\begin{align*}
			& \norm{u^3 - (u^*)^3}_{H^{-1}(\Omega)} \leq C_2 \norm{u^3 - (u^*)^3}_{L^{4/3}(\Omega)} \\
			= & C_2 \norm{(u - u^*) (u^2 + u u^* + (u^*)^2)}_{L^{4/3}(\Omega)} \\
			\leq &  C_2 \norm{(u - u^*) u^2}_{L^{4/3}(\Omega)} +  C_2 \norm{(u - u^*) u u^*}_{L^{4/3}(\Omega)}  C_2 \norm{(u - u^*) (u^*)^2}_{L^{4/3}(\Omega)} \\
			= & C_2 \left(\int_\Omega (u - u^*)^{\frac{4}{3}} u^{\frac{8}{3}}\right)^{\frac{3}{4}} +  C_2 \left(\int_\Omega (u - u^*)^{\frac{4}{3}} u^{\frac{4}{3}} (u^*)^{\frac{4}{3}}\right)^{\frac{3}{4}}  +  C_2 \left(\int_\Omega (u - u^*)^{\frac{4}{3}} (u^*)^{\frac{8}{3}}\right)^{\frac{3}{4}} \\
			\leq & C_2 \norm{u- u^*}_{L^4(\Omega)} \norm{u}_{L^4(\Omega)}^2 + C_2 \norm{u-u^*}_{L^4(\Omega)} \norm{u}_{L^4(\Omega)} \norm{u^*}_{L^4(\Omega)} \\
			&\qquad\qquad + C_2 \norm{u-u^*}_{L^4(\Omega)} \norm{u^*}_{L^4(\Omega)}^2\\
			\leq & C_1^3 C_2 \norm{u-u^*}_{H_0^1(\Omega)} \left(\norm{u}_{H_0^1(\Omega)}^2 +  \norm{u}_{H_0^1(\Omega)} \norm{u^*}_{H_0^1(\Omega)} +  \norm{u^*}_{H_0^1(\Omega)}^2\right)\\
			\leq & L_u \norm{u-u^*}_{H_0^1(\Omega)},
		\end{align*}
		where we have used \eqref{embedL43} and $L_u = C_1^3 C_2  \left(C_u^2 +  C_u \norm{u^*}_{H_0^1(\Omega)} +  \norm{u^*}_{H_0^1(\Omega)}^2\right)$. This finishes the estimation of $\norm{u^3 - (u^*)^3}_{H^{-1}(\Omega)}$. Then we bound $|\gamma - \gamma^*|$. Set
		\begin{equation*}
			A = \left(\calG_{H^1} (V u +\beta |u|^2 u), u\right)_{L^2(\Omega)},\quad A^* = \left(\calG_{H^1} (V u^* + \beta |u^*|^2 u^*), u^*\right)_{L^2(\Omega)},
		\end{equation*}
		and
		\begin{equation*}
			B = \norm{\calG_{H^1} u}_{H_0^1(\Omega)}^2, \quad B^* = \norm{\calG_{H^1} u^*}_{H_0^1(\Omega)}^2.
		\end{equation*}
		One has that
		\begin{align*}
			& |A - A^*|\\
			= & \left| \left(\calG_{H^1} (V u), u\right)_{L^2(\Omega)} - \left(\calG_{H^1} (V u^*), u^*\right)_{L^2(\Omega)}\right| \\
			&\qquad\qquad + \beta \left| \left(\calG_{H^1} (|u|^2u), u\right)_{L^2(\Omega)} - \left(\calG_{H^1} (|u^*|^2 u^*), u^*\right)_{L^2(\Omega)}\right| \\
			\leq & \left| \left(\calG_{H^1} (V u), u -  u^*\right)_{L^2(\Omega)}\right| + \left| \left(\calG_{H^1}  (V u - V u^*), u^*\right)_{L^2(\Omega)} \right| \\
			& \qquad\qquad + \beta\left| \left(\calG_{H^1} (u^3), u -  u^*\right)_{L^2(\Omega)}\right| + \beta\left| \left(\calG_{H^1} (u^3 - (u^*)^3), u^*\right)_{L^2(\Omega)} \right| \\
			\leq & \norm{\calG_{H^1} (V u)}_{L^2(\Omega)} \norm{u -  u^*}_{L^2(\Omega)} + \norm{\calG_{H^1}  (V u - V u^*)}_{L^2(\Omega)} \norm{u^*}_{L^2(\Omega)}\\
			&\qquad\qquad + \beta\norm{\calG_{H^1} (u^3)}_{L^2(\Omega)}\norm{u -  u^*}_{L^2(\Omega)} + \beta\norm{\calG_{H^1} (u^3 - (u^*)^3)}_{L^2(\Omega)}\norm{u^*}_{L^2(\Omega)} \\
			\leq & C_3^2 \norm{\calG_{H^1} (V u)}_{H_0^1(\Omega)} \norm{u -  u^*}_{H_0^1(\Omega)} + C_3 \norm{\calG_{H^1}  (V u - V u^*)}_{H_0^1(\Omega)}\\
			&\qquad\qquad + \beta C_3^2\norm{\calG_{H^1} (u^3)}_{H_0^1(\Omega)}\norm{u -  u^*}_{H_0^1(\Omega)} + \beta C_3\norm{\calG_{H^1} (u^3 - (u^*)^3)}_{H_0^1(\Omega)} \\
			\leq & C_3^3 \norm{V u}_{L^2(\Omega)} \norm{u -  u^*}_{H_0^1(\Omega)} + C_3^2 \norm{V (u - u^*)}_{L^2(\Omega)}\\
			&\qquad\qquad + \beta C_3^2\norm{ u^3}_{H^{-1}(\Omega)}\norm{u -  u^*}_{H_0^1(\Omega)} + \beta C_3\norm{u^3 - (u^*)^3}_{H^{-1}(\Omega)} \\
			\leq & \left( 2 C_3^3 V_{\max} + \beta C_3^2 \norm{u^3}_{H^{-1}(\Omega)} + \beta C_3 L_u \right)\norm{u -  u^*}_{H_0^1(\Omega)} \\
			= & L_A \norm{u -  u^*}_{H_0^1(\Omega)},
		\end{align*}
		where we used \eqref{poincare_ineq}, Lemma~\ref{lem:Gu}~(ii), and
		\begin{equation*}
			\norm{u^3}_{H^{-1}(\Omega)}\leq C_2 \norm{u^3}_{L^{4/3}(\Omega)} = C_2 \norm{u}_{L^4(\Omega)}^3 \leq C_1^3  C_2 \norm{u}_{H_0^1(\Omega)}^3\leq  C_u^3 C_1^3  C_2, 
		\end{equation*}
		and set
		\begin{equation*}
			L_A =  2 C_3^3 V_{\max} + \beta C_3^2 \norm{u^3}_{H^{-1}(\Omega)} + \beta C_3 L_u. 
		\end{equation*}
		Note also that
		\begin{align*}
			|B - B^* | = & \left| \norm{\calG_{H^1} u}_{H_0^1(\Omega)}^2 - \norm{\calG_{H^1} u^*}_{H_0^1(\Omega)}^2 \right| \\
			\leq & \left| \left(\calG_{H^1} u, \calG_{H^1}(u -  u^*)\right)_{H_0^1(\Omega)}\right| + \left| \left(\calG_{H^1} (u - u^*), \calG_{H^1} u^*\right)_{H_0^1(\Omega)} \right| \\
			\leq & \norm{\calG_{H^1} u}_{H_0^1(\Omega)}\norm{\calG_{H^1}(u -  u^*)}_{H_0^1(\Omega)}+\norm{\calG_{H^1} (u - u^*)}_{H_0^1(\Omega)}\norm{\calG_{H^1} u^*}_{H_0^1(\Omega)}\\
			\leq & C_3^2 \norm{ u}_{L^2(\Omega)}\norm{u -  u^*}_{L^2(\Omega)} + C_3^2\norm{u - u^*}_{L^2(\Omega)}\norm{ u^*}_{L^2(\Omega)} \\
			\leq & 2 C_3^3 \norm{u -  u^*}_{H_0^1(\Omega)} =  L_B \norm{u -  u^*}_{H_0^1(\Omega)},
		\end{align*}
		where $L_B = 2 C_3^3$. Then it holds that
		\begin{align*}
			& |\gamma - \gamma^*| = \left|\frac{1 + A}{B} - \frac{1+A^*}{B^*}\right| \leq  \frac{1}{B B^*}\left(|B - B^*| + B^* |A - A^*| + A^* |B - B^*|\right)  \\
			\leq & \frac{1}{ B^* (B^* - L_B \norm{u -  u^*}_{H_0^1(\Omega)})}\left(L_B \norm{u -  u^*}_{H_0^1(\Omega)}  \right. \\ &\qquad\qquad\qquad\left.+ B^* L_A \norm{u -  u^*}_{H_0^1(\Omega)} + A^* L_B \norm{u -  u^*}_{H_0^1(\Omega)}\right)\\
			\leq & \frac{2(L_B +B^* L_A + A^* L_B )}{(B^*)^2}\norm{u -  u^*}_{H_0^1(\Omega)},
		\end{align*}
		for sufficiently small $\norm{u -  u^*}_{H_0^1(\Omega)}$, and the proof is hence completed.
	\end{proof}
	
	\subsection{Proof of Theorem~\ref{thm:local_conv_H1}}

	\begin{proof}[Proof of Theorem~\ref{thm:local_conv_H1}]
		Let us set
		\begin{equation*}
			e_n = u^* - u_n,\quad\text{and}\quad \delta_n = \norm{e_n}_{H_0^1(\Omega)}.
		\end{equation*}
		Let $L_g$ be the constant in Lemma~\ref{lem:Lg} and we assume that $\delta_n$ is small enough such that both Lemma~\ref{lem:Elocalconvex} and Lemma~\ref{lem:Lg} are satisfied. It can be computed that
		\begin{equation}\label{local_esti:1}
			\begin{split}
				& \norm{(u_n - u^*) - \alpha_n \nabla^{\mathcal{R}}_{H^1} E(u_n)}_{H_0^1(\Omega)}^2\\
				= & \norm{u_n - u^*}_{H_0^1(\Omega)}^2 - 2\alpha_n (u_n - u^*, \nabla^{\mathcal{R}}_{H^1}E(u_n))_{H_0^1(\Omega)}  +\alpha_n^2 \norm{\nabla^{\mathcal{R}}_{H^1}E(u_n)}_{H_0^1(\Omega)}^2 \\
				\leq & (1+ L_g^2 \alpha_n^2 )\delta_n^2  + 2\alpha_n (e_n, \nabla^{\mathcal{R}}_{H^1}E(u_n))_{H_0^1(\Omega)}\\
				= & (1+ L_g^2 \alpha_n^2 )\delta_n^2  + 2\alpha_n (e_n, \nabla_{H^1} E(u_n))_{H_0^1(\Omega)} \\ 
				&\qquad\qquad + 2\alpha_n (e_n, \nabla^{\mathcal{R}}_{H^1}E(u_n) - \nabla_{H^1} E(u_n))_{H_0^1(\Omega)}.
			\end{split}
		\end{equation}
		It follows from
		\begin{align*}
			E(u^*) - E(u_n) = & \int_\Omega \left(\frac{1}{2}|\nabla u_n + \nabla e_n|^2 +\frac{1}{2} V |u_n + e_n|^2 +\frac{\beta}{4}|u_n + e_n|^4 \right)\\
			&\qquad\qquad - \int_\Omega \left(\frac{1}{2}|\nabla u_n|^2 + \frac{1}{2} V |u_n|^2 +\frac{\beta}{4} |u_n|^4\right) \\
			= & \int_\Omega \left(\nabla u_n\cdot\nabla e_n + V u_n e_n + \beta |u_n|^2 u_n e_n \right) \\
			&\qquad\qquad + \frac{1}{2}\int_\Omega\left( |\nabla e_n|^2 +  V |e_n|^2 +\beta |u^*|^2 |e_n|^2 \right)\\
			& \qquad\qquad + \int_\Omega \left(\frac{3\beta}{2} |u_n|^2 |e_n|^2 + \beta u_n e_n^3 + \frac{\beta}{4}|e_n|^4 -\frac{\beta}{2} |u^*|^2 |e_n|^2 \right)\\
			=& (e_n, \nabla_{H^1} E(u_n))_{H_0^1(\Omega)} + \frac{1}{2}\int_\Omega \left( |\nabla e_n|^2 +  V |e_n|^2 +\beta |u^*|^2 |e_n|^2\right) \\
			& \qquad\qquad + \int_\Omega \left(\frac{3\beta}{2} |u_n|^2 |e_n|^2 + \beta u_n e_n^3 + \frac{\beta}{4} |e_n|^4 -\frac{\beta}{2} |u^*|^2 |e_n|^2\right),
		\end{align*}
		that
		\begin{equation}\label{local_esti:2}
			\begin{split}
				(e_n, \nabla_{H^1} E(u_n))_{H_0^1(\Omega)}
				= & E(u^*) - E(u_n) - \frac{1}{2}\int_\Omega \left( |\nabla e_n|^2 +  V |e_n|^2 +\beta |u^*|^2 |e_n|^2\right) \\
				&  - \int_\Omega \left(\frac{3\beta}{2} |u_n|^2 |e_n|^2 + \beta u_n e_n^3 + \frac{\beta}{4} |e_n|^4 -\frac{\beta}{2} |u^*|^2 |e_n|^2\right).
			\end{split}
		\end{equation}
		Define 
		\begin{equation*}
			\gamma_n = \frac{1 + \left(\calG_{H^1} (V u_n + \beta u_n^3), u_n\right)_{L^2(\Omega)}}{\norm{\calG_{H^1} u_n}_{H_0^1(\Omega)}^2},
		\end{equation*}
		and let $\gamma^*$ be defined as in \eqref{gammastar}. Then it holds that
		\begin{equation}\label{local_esti:3}
			\begin{split}
				& (e_n, \nabla^{\mathcal{R}}_{H^1}E(u_n) - \nabla_{H^1} E(u_n))_{H_0^1(\Omega)} =  -\gamma_n \left(e_n, \calG_{H^1} u_n\right)_{H_0^1(\Omega)} \\
				=&  -\gamma_n (e_n, u_n)_{L^2(\Omega)} = \frac{\gamma_n}{2}\left(\norm{u_n}_{L^2(\Omega)}^2 + \norm{e_n}_{L^2(\Omega)}^2 - \norm{u_n + e_n}_{L^2(\Omega)}^2\right) \\
				= & \frac{\gamma_n}{2} \norm{e_n}_{L^2(\Omega)}^2.
			\end{split}
		\end{equation}
		Combining \eqref{local_esti:1}, \eqref{local_esti:2}, and \eqref{local_esti:3}, we can conclude that
		\begin{equation}\label{local_esti:4}
			\begin{split}
				& \norm{(u_n - u^*) - \alpha_n \nabla^{\mathcal{R}}_{H^1}E(u_n)}_{H_0^1(\Omega)}^2 \\
				\leq & (1+ L_g^2 \alpha_n^2 )\delta_n^2 +2\alpha_n(E(u^*) - E(u_n))  + \alpha_n \gamma_n \norm{e_n}_{L^2(\Omega)}^2 \\
				& \qquad\qquad - \alpha_n\int_\Omega \left(|\nabla e_n|^2 +  V |e_n|^2 +\beta |u^*|^2 |e_n|^2\right)\\
				&\qquad\qquad - 2\alpha_n \int_\Omega \left(\frac{3\beta}{2} |u_n|^2 |e_n|^2 + \beta |e_n|^2 u_n e_n + \frac{\beta}{4} |e_n|^4 -\frac{\beta}{2} |u^*|^2 |e_n|^2\right).
			\end{split}
		\end{equation}
		We have $\gamma^* = \lambda_0$ since
		\begin{equation*}
			0 = \nabla^{\mathcal{R}}_{H^1}E(u^*) = \calG_{H^1}(-\Delta u^* + V u^* +\beta |u^*|^2 u^*) - \gamma^* \calG_{H^1} u^* = (\lambda_0 - \gamma^*) \calG_{H^1}u^*.
		\end{equation*}
		The minimality of $\lambda_0$ yields that
		\begin{equation}\label{local_esti:5}
			\int_\Omega |\nabla e_n|^2 +  V |e_n|^2 +\beta |u^*|^2 |e_n|^2 \geq \lambda_0\norm{e_n}_{L^2(\Omega)}^2 = \gamma^*\norm{e_n}_{L^2(\Omega)}^2.
		\end{equation}
		It also holds that
		\begin{equation}\label{local_esti:6}
			\begin{split}
				& -\int_\Omega \left( \frac{3\beta}{2} |u_n|^2 |e_n|^2 + \beta |e_n|^2 u_n e_n + \frac{\beta}{4} |e_n|^4 -\frac{\beta}{2} |u^*|^2 |e_n|^2\right) \\
				\leq & - \beta \int_\Omega\left( \frac{1}{2}(|u_n|^2 - |u^*|^2) |e_n|^2 + |e_n|^2 u_n e_n + \frac{1}{4} |e_n|^4\right) \\
				= & - \beta\int_\Omega \left(- \frac{1}{2} |e_n|^4 +\frac{1}{4} |e_n|^4\right) =  \frac{\beta}{4}\norm{e_n}_{L^4(\Omega)}^4 \leq \frac{\beta C_1^4}{4} \delta_n^4.
			\end{split}
		\end{equation}
		We consider sufficiently small $\delta_n$ such that $\gamma_n \leq \gamma^* + \frac{\lambda_1-\lambda_0}{4}$, which can be guaranteed by the proof of Lemma~\ref{lem:Lg}. It thus follows from \eqref{local_esti:4}, \eqref{local_esti:6}, and Lemma~\ref{lem:Elocalconvex} that 
		\begin{align*}
			& \norm{(u_n - u^*) - \alpha_n \nabla^{\mathcal{R}}_{H^1}E(u_n)}_{H_0^1(\Omega)}^2 \\
			\leq & (1+ L_g^2 \alpha_n^2 )\delta_n^2 - \alpha_n\frac{\lambda_1-\lambda_0}{2} \norm{e_n}_{L^2(\Omega)}^2 + \alpha_n \left(\gamma^* + \frac{\lambda_1-\lambda_0}{4}\right) \norm{e_n}_{L^2(\Omega)}^2\\
			& \qquad\qquad - \alpha_n\int_\Omega \left( |\nabla e_n|^2 +  V |e_n|^2 +\beta |u^*|^2 |e_n|^2\right)  + \frac{\alpha_n \beta C_1^4}{2} \delta_n^4\\
			\leq & (1+ L_g^2 \alpha_n^2 )\delta_n^2 + \alpha_n \left(\gamma^* - \frac{\lambda_1-\lambda_0}{4}\right) \norm{e_n}_{L^2(\Omega)}^2 \\
			&\qquad\qquad - \alpha_n\int_\Omega \left( |\nabla e_n|^2 +  V |e_n|^2 +\beta |u^*|^2 |e_n|^2\right)  + \frac{\alpha_n \beta C_1^4}{2} \delta_n^4.
		\end{align*}
		If $\gamma^* - \frac{\lambda_1-\lambda_0}{4}\leq 0$, then 
		\begin{align*}
			&\norm{(u_n - u^*) - \alpha_n \nabla^{\mathcal{R}}_{H^1}E(u_n)}_{H_0^1(\Omega)}^2\\
			\leq& (1+ L_g^2 \alpha_n^2 )\delta_n^2 - \alpha_n\int_\Omega \left(|\nabla e_n|^2 +  V |e_n|^2 +\beta |u^*|^2 |e_n|^2\right) + \frac{\alpha_n \beta C_1^4}{2} \delta_n^4\\
			\leq& (1+ L_g^2 \alpha_n^2 -\alpha_n)\delta_n^2  + \frac{\alpha_n \beta C_1^4}{2} \delta_n^4.
		\end{align*}
		If $\gamma^* - \frac{\lambda_1-\lambda_0}{4}> 0$, then with \eqref{local_esti:5}, it holds that
		\begin{align*}
			& \norm{(u_n - u^*) - \alpha_n \nabla^{\mathcal{R}}_{H^1}E(u_n)}_{H_0^1(\Omega)}^2 \\
			\leq & (1+ L_g^2 \alpha_n^2 )\delta_n^2 + \alpha_n \left(\gamma^* - \frac{\lambda_1-\lambda_0}{4}\right) \frac{1}{\gamma^*} \int_\Omega \left(|\nabla e_n|^2 +  V |e_n|^2 +\beta |u^*|^2 |e_n|^2\right) \\
			&\qquad\qquad - \alpha_n\int_\Omega \left(|\nabla e_n|^2 +  V |e_n|^2 +\beta |u^*|^2 |e_n|^2\right)  + \frac{\alpha_n \beta C_1^4}{2} \delta_n^4\\
			\leq & (1+ L_g^2 \alpha_n^2 )\delta_n^2 - \alpha_n \frac{\lambda_1 - \lambda_0}{4\gamma^*} \int_\Omega \left(|\nabla e_n|^2 +  V |e_n|^2 +\beta |u^*|^2 |e_n|^2\right) + \frac{\alpha_n \beta C_1^4}{2} \delta_n^4\\
			\leq & \left(1+ L_g^2 \alpha_n^2 - \alpha_n \frac{\lambda_1 - \lambda_0}{4\lambda_0} \right)\delta_n^2 + \frac{\alpha_n \beta C_1^4}{2} \delta_n^4.
		\end{align*}
		In both cases we have
		\begin{multline*}
			\norm{(u_n - u^*) - \alpha_n \nabla^{\mathcal{R}}_{H^1}E(u_n)}_{H_0^1(\Omega)}^2 \\
			\leq \left(1+ L_g^2 \alpha_n^2 - \alpha_n \min\left\{1,\frac{\lambda_1 - \lambda_0}{4\lambda_0}\right\}  + \frac{\alpha_n \beta C_1^4}{2} \delta_n^2\right)\delta_n^2,
		\end{multline*}
		i.e.,
		\begin{multline*}
			\norm{(u_n - u^*) - \alpha_n \nabla^{\mathcal{R}}_{H^1}E(u_n)}_{H_0^1(\Omega)}  \\
			\leq \left(1+ L_g^2 \alpha_n^2 - \alpha_n \min\left\{1,\frac{\lambda_1 - \lambda_0}{4\lambda_0}\right\}  + \frac{\alpha_n \beta C_1^4}{2} \delta_n^2\right)^{1/2}\delta_n.
		\end{multline*}
		Therefore, with $R_n$ defined in \eqref{def:Rn}, it holds that
		\begin{align*}
			\delta_{n+1} \leq & \norm{(u_n - u^*) - \alpha_n \nabla^{\mathcal{R}}_{H^1}E(u_n)}_{H_0^1(\Omega)} + \norm{R_n}_{H_0^1(\Omega)} \\
			\leq &  \left(1+ L_g^2 \alpha_n^2 - \alpha_n \min\left\{1,\frac{\lambda_1 - \lambda_0}{4\lambda_0}\right\}  + \frac{\alpha_n \beta C_1^4}{2} \delta_n^2\right)^{1/2}\delta_n  \\
			&\qquad\qquad + \frac{\alpha_n^2}{2} (C_u + \alpha_n C_g) C_3^2 L_g^2 \delta_n^2,
		\end{align*}
		where $\norm{R_n}_{H_0^1(\Omega)}$ is established using \eqref{eq:Rn} and Lemma~\ref{lem:Lg}. Note that we have assumed \eqref{asp:alpha_local}, that guarantees that
		\begin{equation*}
			\max_{\alpha\in[\alpha_{\min},\alpha_{\max}]} \left\{ 1+ L_g^2 \alpha^2 - \alpha \min\left\{1,\frac{\lambda_1 - \lambda_0}{4\lambda_0}\right\} \right\} < 1.
		\end{equation*}
		Thus $\delta_n$ sufficiently small, we can conclude that $\delta_{n+1}\leq C_\delta \delta_n$, where $C_\delta\in(0,1)$ is a constant. This proves the locally exponentially convergent rate.
	\end{proof}
	
	\subsection*{Acknowledgement}
	The work of ZC and JL is supported in part by National Science Foundation via awards DMS-2012286 and DMS-2309378. 
	YL thanks NSF for the support via the award DMS-2343135. XZ is supported by NSF  DMS-2208518.
	\bibliographystyle{amsxport}
	\bibliography{references}
	
	\appendix
	
	\section{Equivalent Norms}
	
	\begin{lemma}[Equivalence between $\norm{\cdot}_{H_0^1(\Omega)}$ and $\norm{\cdot}_{a_0(\Omega)}$]
		\label{lem:equiv_a0_H1}
		For any $u\in H_0^1(\Omega)$, it holds that
		\begin{equation*}
			\norm{u}_{H_0^1(\Omega)}\leq \norm{u}_{a_0(\Omega)} \leq C_{a_0} \norm{u}_{H_0^1(\Omega)},
		\end{equation*}
		where $C_{a_0} = \left(1 +  C_3^2 V_{\max}\right)^{1/2}$ and $C_3$ is the constant in the Poincaré inequality \eqref{poincare_ineq}.
	\end{lemma}
	
	\begin{proof}
		Use the definition of $\norm{\cdot}_{H_0^1(\Omega)}$ and $\norm{\cdot}_{a_0(\Omega)}$ as well as the Poincaré inequality.
	\end{proof}
	
	\begin{lemma}[Equivalence between $\norm{\cdot}_{H_0^1(\Omega)}$ and $\norm{\cdot}_{a_{u^*}(\Omega)}$]\label{lem:equiv_au_H1}
		Let $u^*$ be the ground state of the problem \eqref{eigen_problem}. For any $u\in H_0^1(\Omega)$, it holds that
		\begin{equation*}
			\norm{u}_{H_0^1(\Omega)}\leq \norm{u}_{a_{u^*}(\Omega)} \leq C_{a_{u^*}} \norm{u}_{H_0^1(\Omega)},
		\end{equation*}
		where $C_{a_{u^*}} = \left(1+ V_{\max} C_3^2 + \beta C_1^2 \norm{u^*}_{L^4(\Omega)}^2 \right)^{1/2}$, where $C_1$ and $C_3$ are the constants in \eqref{embed_L4} and \eqref{poincare_ineq}, respectively.
	\end{lemma}
	
	\begin{proof}
		It holds that
		\begin{align*}
			\norm{u}_{H_0^1(\Omega)}^2 & \leq \norm{u}_{a_{u^*}(\Omega)}^2  = \int_\Omega |\nabla u|^2 + V|u|^2 + \beta |u^*|^2 |u|^2 \\
			& \leq \norm{u}_{H_0^1(\Omega)}^2 + V_{\max} \norm{u}_{L^2(\Omega)}^2 + \beta \norm{u^*}_{L^4(\Omega)}^2 \norm{u}_{L^4(\Omega)}^2 \\
			& \leq \left(1+ V_{\max} C_3^2 + \beta C_1^2 \norm{u^*}_{L^4(\Omega)}^2 \right) \norm{u}_{H_0^1(\Omega)}^2.
		\end{align*}
	\end{proof}
	
	\begin{lemma}[Stability of $\norm{\cdot}_{a_u(\Omega)}$ around $u^*$]\label{lem:stab_au}
		Let $u^*$ be the ground state of the problem \eqref{eigen_problem}. For any $\epsilon>0$, there exists $\delta(\epsilon)>0$, such that
		\begin{equation*}
			\frac{1}{1+\epsilon} \norm{z}_{a_{u}}\leq \norm{z}_{a_{u^*}} \leq (1+\epsilon) \norm{z}_{a_{u}},\quad \forall~z\in H^1(\Omega),
		\end{equation*}
		holds for any $u\in H_0^1(\Omega)$ with $\norm{u - u^*}_{H_0^1(\Omega)}\leq \delta(\epsilon)$.
	\end{lemma}
	
	\begin{proof}
		We have that
		\begin{align*}
			\left|\norm{z}_{a_{u^*}}^2 - \norm{z}_{a_u}^2\right| & = \beta \left|\int_\Omega (|u|^2 - |u^*|^2)\cdot |z|^2 \right| \leq \beta\int_\Omega |u - u^*| \cdot|u + u^*|\cdot |z|^2\\
			&\leq \beta \norm{u - u^*}_{L^4(\Omega)} \norm{u + u^*}_{L^4(\Omega)} \norm{z}_{L^4(\Omega)}^2 \\
			& \leq \beta C_1^4 \norm{u - u^*}_{H_0^1(\Omega)} \left(2\norm{u^*}_{H_0^1(\Omega)} + \norm{u - u^*}_{H_1^0(\Omega)}\right)\norm{z}_{a_{u^*}(\Omega)}^2.
		\end{align*}
		Thus, it suffices to set $\delta(\epsilon)$ as small enough.
	\end{proof}

	\section{Proof of Theorem~\ref{thm:local_conv_au}}
	\label{sec:pf_local_au}
	
	We display the proof of Theorem~\ref{thm:local_conv_au} in this section. According to \cite{henning2020sobolev}*{Lemma 4.7}, the iterates $\{u_n\}_{n=0}^\infty$ generated by the $a_u$-scheme yields energy decay and hence $\norm{u_n}_{H_0^1(\Omega)}\leq C_u$ as long as $\alpha_{\max}\leq C_\alpha$, where $C_\alpha$ and $C_u$ are constants depending only on $\Omega$, $d$, $\beta$, $V$, and $\norm{u_0}_{H_0^1(\Omega)}$. We would need a similar result as Lemma~\ref{lem:Lg}, for which we prove the following lemma.
	
	\begin{lemma}\label{lem:Gau}
		Let $u^*$ be the ground state of the problem \eqref{eigen_problem}. For any $u\in\calM$ with $\norm{u}_{H_0^1(\Omega)}\leq C_u$, it holds that
		\begin{equation*}
			\norm{\calG_{a_u} u - \calG_{a_{u^*}} u^*}_{a_{u^*}(\Omega)} \leq L_{\calG} \norm{u - u^*}_{a_{u^*}(\Omega)},
		\end{equation*}
		where $L_\calG =  \left(C_3^2  + \beta C_1^4 C_3 \left(C_u + \norm{u^*}_{H_0^1(\Omega)}\right)\right)$.
	\end{lemma}
	
	\begin{proof}
		Denote $g = \calG_{a_u} u$ and $g^* = \calG_{a_{u^*}} u^*$.  Then
		\begin{equation*}
			-\Delta g + V g + \beta |u|^2 g = u,\quad  -\Delta g^* + V g^* + \beta |u^*|^2 g^* = u^*,
		\end{equation*}
		which lead to
		\begin{equation*}
			-\Delta(g - g^*) + V(g - g^*) + \beta |u^*|^2 (g - g^*) = (u - u^*) + \beta (|u^*|^2 - |u|^2) g.
		\end{equation*}
		Therefore,
		\begin{align*}
			\norm{g - g^*}_{a_{u^*}(\Omega)}^2 & = \int_\Omega(u - u^*)(g - g^*) + \int_\Omega \beta (|u^*|^2 - |u|^2) g(g - g^*) \\
			& \leq \norm{u-u^*}_{L^2(\Omega)} \norm{g - g^*}_{L^2(\Omega)}  \\
			& \qquad\quad + \beta \norm{u-u^*}_{L^4(\Omega)}\norm{u + u^*}_{L^4(\Omega)} \norm{g}_{L^4(\Omega)} \norm{g - g^*}_{L^4(\Omega)}\\
			& \leq C_3^2 \norm{u- u^*}_{a_{u^*}(\Omega)} \norm{g - g^*}_{a_{u^*}(\Omega)} \\
			&\qquad\quad + \beta C_1^4 \norm{u - u^*}_{H_0^1(\Omega)}\norm{u + u^*}_{H_0^1(\Omega)} \norm{g}_{H_0^1(\Omega)} \norm{g - g^*}_{H_0^1(\Omega)},
		\end{align*}
		where we used \eqref{poincare_ineq} and \eqref{embed_L4}. This implies that 
		\begin{equation*}
			\norm{g - g^*}_{a_{u^*}(\Omega)} \leq \left(C_3^2  + \beta C_1^4 \left(C_u + \norm{u^*}_{H_0^1(\Omega)}\right)\norm{g}_{H_0^1(\Omega)}\right) \norm{u - u^*}_{a_{u^*}(\Omega)}.
		\end{equation*}
		Then we can obtain the desired results by noticing
		\begin{equation}\label{local_esti:7}
			\norm{\calG_{a_u}u}_{H_0^1(\Omega)} = \norm{g}_{H_0^1(\Omega)} \leq C_3,
		\end{equation}
		from
		\begin{multline*}
			\norm{g}_{H_0^1(\Omega)}^2 \leq \norm{g}_{a_u(\Omega)}^2 = (\calG_{a_u} u, g )_{a_u(\Omega)} \\
			= (u, g)_{L^2(\Omega)} \leq  \norm{u}_{L^2(\Omega)} \norm{g}_{L^2(\Omega)} \leq C_3 \norm{g}_{H_0^1(\Omega)}.
		\end{multline*}
	\end{proof}
	
	The next lemma establishes a similar estimate as in Lemma~\ref{lem:Lg}.
	
	\begin{lemma}\label{lem:Lg_au}
		There exists some constant $L_g$ depending only on $\Omega$, $d$, $\beta$, $V$, $u^*$ and $\norm{u_0}_{H_0^1(\Omega)}$ such that 
		\begin{equation*}
			\norm{\nabla^{\mathcal{R}}_{a_u}E(u)}_{a_{u^*}(\Omega)}\leq L_g\norm{u - u^*}_{a_{u^*}(\Omega)},
		\end{equation*}
		holds for all $u\in \calM$ as long as $\norm{u}_{H_0^1(\Omega)}\leq C_u$ and $\norm{u - u^*}_{a_{u^*}(\Omega)}$ is sufficiently small.
	\end{lemma}
	
	\begin{proof}
		Denote
		\begin{equation*}
			\gamma = \frac{1}{\norm{\calG_{a_u} u }_{a_u(\Omega)}^2} = \frac{1}{\left(\calG_{a_u} u, u\right)_{L^2(\Omega)}},
		\end{equation*}
		and
		\begin{equation*}
			\gamma^* = \frac{1}{\norm{\calG_{a_{u^*}} u^* }_{a_{u^*}(\Omega)}^2} = \frac{1}{\left(\calG_{a_{u^*}} u^*, u^*\right)_{L^2(\Omega)}}.
		\end{equation*}
		By \eqref{poincare_ineq}, \eqref{local_esti:7}, and Lemma~\ref{lem:Gau}, one has that
		\begin{align*}
			& \left|\left(\calG_{a_u} u, u\right)_{L^2(\Omega)} - \left(\calG_{a_{u^*}} u^*, u^*\right)_{L^2(\Omega)} \right| \\
			\leq & \left| \left(\calG_{a_u} u, u - u^*\right)_{L^2(\Omega)}\right| + \left| \left(\calG_{a_u} u - \calG_{a_{u^*}} u^*, u^*\right)_{L^2(\Omega)} \right| \\
			\leq & \norm{\calG_{a_u} u}_{L^2(\Omega)} \norm{u - u^*}_{L^2(\Omega)} + \norm{\calG_{a_u} u - \calG_{a_{u^*}} u^*}_{L^2(\Omega)} \norm{u^*}_{L^2(\Omega)}\\
			\leq & C_3^2 \norm{\calG_{a_u} u}_{H_0^1(\Omega)} \norm{u - u^*}_{a_{u^*}(\Omega)} + C_3 \norm{\calG_{a_u} u - \calG_{a_{u^*}} u^*}_{a_{u^*}(\Omega)} \\
			\leq & \left(C_3^3 + C_3 L_\calG\right) \norm{u - u^*}_{a_{u^*}(\Omega)},
		\end{align*}
		and hence that
		\begin{equation*}
			|\gamma_n - \gamma^*|  = \frac{\left|\left(\calG_{a_u} u, u\right)_{L^2(\Omega)} - \left(\calG_{a_{u^*}} u^*, u^*\right)_{L^2(\Omega)} \right|}{\left(\calG_{a_u} u, u\right)_{L^2(\Omega)} \left(\calG_{a_{u^*}} u^*, u^*\right)_{L^2(\Omega)}}\leq L_\gamma \norm{u - u^*}_{a_{u^*}(\Omega)},
		\end{equation*}
		for some constant $L_\gamma$ and sufficiently small $\norm{u - u^*}_{a_{u^*}(\Omega)}$. Then it holds that
		\begin{align*}
			&\norm{\nabla_{a_u}^{\mathcal{R}}E(u)}_{a_{u^*}(\Omega)} \\
			= & \norm{\nabla_{a_u}^{\mathcal{R}}E(u) - \nabla_{a_{u^*}}^{\mathcal{R}}E(u^*)}_{a_{u^*}(\Omega)}= \left\|u  - \gamma \calG_{a_u} u  - u^* + \gamma^* \calG_{a_{u^*}} u^*\right\|_{a_{u^*}(\Omega)}\\
			\leq & \norm{u - u^*}_{a_{u^*}(\Omega)} + |\gamma - \gamma^*|\cdot \norm{\calG_{a_u} u}_{a_{u^*}(\Omega)} + \gamma^* \norm{\calG_{a_u} u - \calG_{a_{u^*}} u^*}_{a_{u^*}} \\
			\leq & (1 + L_\gamma C_{a_{u^*}} C_3 + \gamma^* L_\calG) \norm{u - u^*}_{a_{u^*}(\Omega)},
		\end{align*}
		where we used \eqref{local_esti:7}, Lemma~\ref{lem:equiv_au_H1}, and Lemma~\ref{lem:Gau}.
	\end{proof}

	\begin{proof}[Proof of Theorem~\ref{thm:local_conv_au}]
		Denote
		\begin{equation*}
			e_n = u^* - u_n,\quad \delta_n = \norm{e_n}_{a_{u^*}(\Omega)}, \quad\text{and}\quad \gamma_n = \frac{1}{\left(\calG_{a_{u_n}} u_n, u_n\right)_{L^2(\Omega)}}.
		\end{equation*}
		We assume that $u_n$ is close enough in $\norm{\cdot}_{H_0^1(\Omega)}$ to $u^*$ so that the results in Lemma~\ref{lem:stab_au} and Lemma~\ref{lem:Lg_au} are true with sufficiently small $\epsilon$ satisfying
		\begin{equation*}
			\max_{\alpha\in[\alpha_{\min},\alpha_{\max}]}(1+\epsilon)^2 \left((1+\epsilon)^2(1+ L_g^2 \alpha^2) - \alpha \min\left\{1,\frac{\lambda_1 - \lambda_0}{4\lambda_0}\right\} \right) < 1.
		\end{equation*}
		It can be computed that
		\begin{align*}
			& \norm{(u_n - u^*) - \alpha_n \nabla_{a_{u_n}}^{\mathcal{R}} E(u_n)}_{a_{u_n}(\Omega)}^2 \\
			\leq & \norm{u_n - u^*}_{a_{u_n}(\Omega)}^2 - 2\alpha_n \left(u_n - u^*, \nabla_{a_{u_n}}^{\mathcal{R}}E(u_n)\right)_{a_{u_n}(\Omega)}  +\alpha_n^2 \norm{\nabla_{a_{u_n}}^{\mathcal{R}}E(u_n)}_{a_{u_n}(\Omega)}^2 \\
			\leq & (1+\epsilon)^2(1+ L_g^2 \alpha_n^2 )\delta_n^2  + 2\alpha_n (e_n, \nabla_{a_{u_n}}^{\mathcal{R}}E(u_n))_{a_{u_n}}\\
			= & (1+\epsilon)^2 (1+ L_g^2 \alpha_n^2 )\delta_n^2  + 2\alpha_n (e_n, u_n )_{a_{u_n}(\Omega)}  - 2\alpha_n\gamma_n \left(e_n, \calG_{a_{u_n}} u_n \right)_{a_{u_n}(\Omega)}.
		\end{align*}
		By some computations similar to those in the proof of Theorem~\ref{thm:local_conv_H1}, it holds for sufficiently small $\delta_n$ that
		\begin{multline*}
			\norm{(u_n - u^*) - \alpha_n \nabla_{a_{u_n}}^{\mathcal{R}}E(u_n)}_{a_{u_n}(\Omega)}^2 \\ \leq \left((1+\epsilon)^2(1+ L_g^2 \alpha_n^2) - \alpha_n \min\left\{1,\frac{\lambda_1 - \lambda_0}{4\lambda_0}\right\}  + \frac{\alpha_n \beta C_1^4}{2} \delta_n^2\right)\delta_n^2,
		\end{multline*}
		and hence that
		\begin{multline*}
			\norm{(u_n - u^*) - \alpha_n \nabla_{a_{u_n}}^{\mathcal{R}}E(u_n)}_{a_{u^*}(\Omega)} \\ \leq (1+\epsilon)\left((1+\epsilon)^2(1+ L_g^2 \alpha_n^2) - \alpha_n \min\left\{1,\frac{\lambda_1 - \lambda_0}{4\lambda_0}\right\}  + \frac{\alpha_n \beta C_1^4}{2} \delta_n^2\right)^{1/2}\delta_n.
		\end{multline*}
		Note that
		\begin{equation*}
			u_{n+1} = R\left(u_n - \alpha_n \nabla_{a_{u_n}}^{\mathcal{R}}E(u_n)\right) = \left(u_n - \alpha_n \nabla_{a_{u_n}}^{\mathcal{R}}E(u_n)\right) + R_n, 
		\end{equation*}
		where 
		\begin{equation*}
			R_n = \left(u_n - \alpha_n \nabla_{a_{u_n}}^{\mathcal{R}}E(u_n)\right) - R\left(u_n - \alpha_n \nabla_{a_{u_n}}^{\mathcal{R}}E(u_n)\right),
		\end{equation*} 
		can be estimated similar to Lemma \ref{lem:esti_retraction} that
		\begin{align*}
			\norm{R_n}_{a_{u^*}(\Omega)} & \leq \frac{\alpha_n^2}{2}\norm{\nabla_{a_{u_n}}^{\mathcal{R}}E(u_n)}_{L^2(\Omega)}^2 \norm{u_n - \alpha_n \nabla_{a_{u_n}}^{\mathcal{R}}E(u_n)}_{a_{u^*}(\Omega)} \\
			& \leq \frac{\alpha_n^2}{2} L_g^2\delta_n^2 (C_{a_{u^*}} C_u + \alpha_n L_g \delta_n).
		\end{align*}
		Then we can conclude the locally exponential convergence rate via
		\begin{align*}
			\delta_{n+1} \leq & \norm{(u_n - u^*) - \alpha_n \nabla_{a_{u_n}}^{\mathcal{R}}E(u_n)}_{a_{u^*}} + \norm{R_n}_{a_{u^*}} \\
			\leq &  (1+\epsilon)\left((1+\epsilon)^2(1+ L_g^2 \alpha_n^2) - \alpha_n \min\left\{1,\frac{\lambda_2 - \lambda_1}{4\lambda^*}\right\}  + \frac{\alpha_n \beta C_1^4}{2} \delta_n^2\right)^{1/2}\delta_n  \\
			&\qquad\qquad + \frac{\alpha_n^2}{2} L_g^2\delta_n^2 (C_{a_{u^*}} C_u + \alpha_n L_g \delta_n)\\
			\leq & C_\delta \delta_n,
		\end{align*}
		where $C_\delta\in(0,1)$ and $\delta_n$ is sufficiently small.
	\end{proof}
\end{document}